\documentclass[a4paper,11pt]{amsart}
\usepackage{amssymb,amscd,amsxtra,xypic}
\usepackage[all]{xy}

\usepackage[pdftex]{hyperref}
\hypersetup{%
bookmarksnumbered=true,%
colorlinks=true,%
setpagesize=false,%
pdftitle={},%
pdfauthor={}}

\theoremstyle{plain}
\newtheorem{Thm}{Theorem}[section]
\newtheorem{Lem}[Thm]{Lemma}
\newtheorem{Cor}[Thm]{Corollary}
\newtheorem{Prop}[Thm]{Proposition}
\newtheorem{Conj}[Thm]{Conjecture}

\theoremstyle{definition}
\newtheorem{Def}[Thm]{Definition}
\newtheorem{Def-Lem}[Thm]{Definition-Lemma}
\newtheorem{Cond}[Thm]{Condition}
\newtheorem{Rem}[Thm]{Remark}

\newtheorem*{Ack}{Acknowledgments}
\newtheorem{Ex}[Thm]{Example}

\newcommand{\Aut}{\operatorname{Aut}}
\newcommand{\Bir}{\operatorname{Bir}}

\newcommand{\prt}{\partial}
\newcommand{\Sing}{\operatorname{Sing}}
\newcommand{\Spec}{\operatorname{Spec}}

\newcommand{\Cl}{\operatorname{Cl}}

\newcommand{\Int}{\operatorname{Int}}
\newcommand{\bNE}{\operatorname{\overline{NE}}}
\newcommand{\Bs}{\operatorname{Bs}}

\newcommand{\Exc}{\operatorname{Exc}}
\newcommand{\mult}{\operatorname{mult}}

\newcommand{\Eff}{\operatorname{Eff}}

\newcommand{\ord}{\operatorname{ord}}
\newcommand{\lct}{\operatorname{lct}}
\newcommand{\Supp}{\operatorname{Supp}}

\newcommand{\mbA}{\mathbb{A}}

\newcommand{\mbC}{\mathbb{C}}

\newcommand{\mbP}{\mathbb{P}}
\newcommand{\mbQ}{\mathbb{Q}}

\newcommand{\mbZ}{\mathbb{Z}}

\newcommand{\mcF}{\mathcal{F}}

\newcommand{\mcI}{\mathcal{I}}

\newcommand{\mcL}{\mathcal{L}}
\newcommand{\mcM}{\mathcal{M}}

\newcommand{\mcO}{\mathcal{O}}

\newcommand{\mcW}{\mathcal{W}}

\newcommand{\msp}{\mathsf{p}}
\newcommand{\msq}{\mathsf{q}}

\newcommand{\ratmap}{\dashrightarrow}
\newcommand{\lrd}{\llcorner}
\newcommand{\rrd}{\lrcorner}

\makeatletter
\def\imod#1{\allowbreak\mkern10mu({\operator@font mod}\,\,#1)}
\makeatother

\title[Alpha invariants of birationally rigid Fano threefolds]{Alpha invariants of birationally rigid\\ Fano threefolds}
\author{In-Kyun~Kim \and Takuzo~Okada \and Joonyeong~Won}
\address{Department of Mathematics, Seoul National University, Seoul 151-747, Korea}
\email{soulcraw@gmail.com}
\address{Department of Mathematics, Faculty of Science and Engineering, Saga University, Saga 840-8501 Japan}
\email{okada@cc.saga-u.ac.jp}
\address{Center for Geometry and Physics, Institute for Basic Science (IBS) 77 Cheongam-ro, Nam-gu, Pohang, Gyeongbuk, 37673, Korea}
\email{leonwon@kias.re.kr}
\date{}

\begin{document}

\begin{abstract}
We compute global log canonical thresholds, or equivalently alpha invariants, of birationally rigid orbifold Fano threefolds embedded in weighted projective spaces as codimension two or three.
As an important application, we prove that most of them are weakly exceptional, $K$-stable and admit K\"{a}hler--Einstien metric.
\end{abstract}

\maketitle


\section{Introduction} \label{sec:intro}

All considered varieties are assumed to be algebraic and defined over the complex number field $\mbC$ throughout this article. 

The aim of this article is to compute global log canonical thresholds $\lct (X)$, or equivalently alpha-invariants $\alpha (X)$ of Tian \cite{Ti87}, of birationally rigid orbifold Fano $3$-folds $X$ of low degree.
The {\it global log canonical threshold} $\lct (X)$ of a Fano variety $X$ is defined to be the supremum of log canonical thresholds of the pairs $(X,D)$, where $D$ runs over the effective $\mbQ$-divisors on $X$ that are numerically equivalent to $-K_X$.
The main application of these computations is to show the existence of K\"{a}hler-Einstein metric on those birationally rigid Fano $3$-folds: It follows from \cite{DK01,Na90,Ti87} that an orbifold Fano variety $X$ admits an orbifold K\"{a}hler-Einstein metric if $\lct (X) > \dim X/(\dim X+1)$.

By the results of Chen-Donaldson-Sun \cite{cds1,cds2,cds3}, the existence of K\"{a}hler-Einstein metric for a manifold is equivalent to the algebro-geometric notion of $K$-stability, a version of stability notion of Geometric Invariant Theory. 
However it is still hard to determine $K$-stability of a given algebraic variety and computing global log canonical threshold is one of a few ways to show the existence of K\"{a}hler--Einstein metric.
We refer the reader to \cite{OO} for an algebro-geometric argument that proves that a Fano variety satisfying $\lct (X) > \dim X/(\dim X + 1)$ is $K$-stable.

From now on, by a Fano variety, we mean a normal projective $\mbQ$-factorial variety $X$ with only terminal singularities such that $-K_X$ is ample. 
With this terminology, a Fano $3$-fold $X$ of Picard number one, together with the structure morphism $X \to \Spec \mbC$, can be thought of as a Mori fiber space.
A Fano variety $X$ is called {\it birationally rigid} if $X$ is not birational to a Mori fiber space other than $X$ itself.
The first example of birationally rigid Fano variety is the smooth quartic threefolds provided by Iskovskikh and Manin \cite{IM}. 
Then it was extended to orbifold Fano $3$-folds: $95$ families of quasi-smooth Fano threefold weighted hypersurface of index one are birationally rigid \cite{CPR,CP}. 
For the global log canonical thresholds of these $95$ families, we have the following result.

\begin{Thm}[\cite{Ch08}] \label{thm:Ch}
Let X be a general quasi-smooth Fano weighted hypersurface of index one.
If $-K_X^3\leq 1$, then $\lct (X) = 1$.
\end{Thm}

We move on to higher codimensional cases.
There are $85$ and $70$ families of quasi-smooth Fano $3$-folds of index one embedded in weighted projective spaces as subvarieties of codimension $2$ and $3$, respectively.
The family of complete intersections of a quadric and cubic in $\mbP^5$ is the unique family of nonsingular Fano $3$-folds of index one and of codimension $2$, and a general such variety is proved birationally rigid \cite{IP}.
We have the following result for the remaining index one Fano $3$-folds of codimension $2$ and $3$.

\begin{Thm}[\cite{Okada1}, \cite{AZ}] 
Let $X$ be a quasi-smooth singular Fano $3$-fold of index one embedded in a weighted projective space as a codimension two subvariety. 
Then X is birationally rigid if and only if it belongs to one of $18$ specific families among $84$ families.
\end{Thm}

\begin{Thm}[\cite{AO}] 
Let $X$ be a general quasi-smooth Fano $3$-fold of index one embedded in a weighted projective space as a codimension $3$ subvariety. 
Then X is birationally rigid if and only if it belongs to one of $3$ specific families among $70$ families.
\end{Thm}

We state the main theorem of this article, which completes the computation of the global log canonical thresholds of all the birationally rigid Fano $3$-folds of codimension $2$ and $3$ except for the complete intersections of a quadric and a cubic in $\mbP^5$.

\begin{Thm} \label{mainthm}
Let $X$ be a general quasi-smooth Fano $3$-fold of index one embedded in a weighted projective space as a codimension $c \in \{2,3\}$ subvariety.
\begin{enumerate}
\item When $c = 2$. If $X$ is birationally rigid and it is not a complete intersection of quadric and cubic in $\mbP^3$, then $\lct (X) \ge 1$.
Moreover the equality $\lct (X) = 1$ holds if $X$ is not a member of family No.~$60$ $($cf.\ \emph{Section \ref{sec:birrigcodim2}}$)$.
\item When $c = 3$. If $X$ is birationally rigid, then $\lct (X)=1$.
\end{enumerate}
\end{Thm}

As a direct consequence, we have the following.

\begin{Cor}
With the assumptions of \emph{Theorem \ref{mainthm}}, the Fano $3$-fold $X$ has a K\"{a}hler--Einstein metric and is K-stable.
\end{Cor}

The inequality $\lct (X) > \dim X / (\dim X + 1)$ is enough to conclude the existence of a K\"{a}hler--Einstein metric on a Fano variety $X$.
However, the sharp estimate $\lct (X) \ge 1$ is essential in the following application (see \cite[Corollary 1.5, Theorem 6.5]{Ch08} and \cite{Pu05}).

\begin{Cor}
Let $X_1, \dots,X_r$ be varieties which satisfy the hypotheses of \emph{Theorem \ref{mainthm}} and let $V = X_1 \times \cdots X_r$.
Then
\[
\Bir (V) = \left\langle \prod_{i=1}^r \Bir (X_i), \Aut (V) \right\rangle,
\]
the variety $V$ is non-rational, and for any dominant map $\rho \colon V \ratmap Y$ whose general fiber is rationally connected, there is a commutative diagram
\[
\xymatrix{
V \ar[d]_{\pi} \ar@{-->}[r]^{\sigma} & V \ar[d]^{\rho} \\
X_{i_1} \times \cdots X_{i_k} \ar@{-->}[r]_{\hspace{8mm} \xi} & Y,}
\]
where $\xi$ and $\sigma$ are birational maps, and $\pi$ is a projection for some $\{i_1,\dots,i_k\} \subsetneq \{1,\dots,r\}$.
\end{Cor}

We explain an another application concerning singularities of the origin of the affine cone over varieties. We say that a singularity $(o \in V)$ is {\it weakly exceptional} if it has a unique plt blow up. 
It is a natural generalization of surface singularity of type $\mathbb{D}$ and $\mathbb{E}$.
Let $o \in V$ be a germ of a Kawamata log terminal singularity. Then there exists a birational morphism $\pi \colon W \rightarrow V$ with a single irreducible divisor $E \subset W$ such that $o \in \pi (E)$ and $\pi$ is a plt blowup. 
The papers \cite{Ku02,Pr00} state that $o \in V$ is weakly exceptional if and only if $\pi(E) = o$ and the log pair $(E, \mathrm{Diff}_E (0) + D)$ is log canonical for every effective $\mathbb{Q}$-divisor $D$ on the variety $E$ such that $D \sim_{\mathbb{Q}} -(K_E +\mathrm{Diff}_E(0))$.
Here $\mathrm{Diff}_E (0)$ is called the {\it different} and it satisfies $K_E + \mathrm{Diff}_E (0) \equiv (K_W + E)|_E$.
Thus for an index one Fano variety $X$ and the corresponding affine cone $V$ over $X$, the exceptional locus of the blowup at the origin of $o \in V$ coincides with the underlying variety $X$. 
We say that an index one Fano variety $X$ is {\it weakly exceptional} if, the origin of the affine cone $o \in V$ over $X$ is weakly exceptional.
With this notation, $X$ is weakly exceptional if $\lct (X) \ge 1$.
Thus we have the following.

\begin{Cor}
With the assumptions of \emph{Theorem \ref{mainthm}}, the Fano $3$-fold $X$ is weakly exceptional.
\end{Cor}

We finish the introduction by posing a conjecture.
To prove birational rigidity and to compute global log canonical threshold resemble each other although the details are much different.
Nowadays many examples of birationally (super)rigid Fano varieties have been known while we do not have exact values of global log canonical thresholds for some of them.
However we believe that they surely have similar nature and we come to the following conjecture. 

\begin{Conj}  
Birationally rigid (orbifold) Fano varieties are K-stable and admit K\"{a}hler--Einstein metric.
\end{Conj}

The main result of this article, Theorem \ref{thm:Ch} by Cheltsov and  series of results in \cite{CPW,EP,Pu05} provide us with supporting evidences for this conjecture.
A more conceptual evidence is given by Odaka and Okada \cite{OO}, where it is proved that birational super-rigidity (with additional mild assumptions) implies slope stability, a weaker version of $K$-stability. 

\begin{Ack}
Part of this work is done during the authors' stay at IBS Center for Geometry and Physics in Korea.
The authors would like to thank the institute, and especially Professor Jihun Park, for their hospitality.
The second author is partially supported by JSPS KAKENHI Grant Number 26800019, and the third author by NRF-2014R1A1A2056432, the National Foundation in Korea.
\end{Ack}

\section{Preliminaries} \label{sec:prelim}

Let $X$ be a Fano variety, i.e.\ a normal projective $\mbQ$-factorial variety with at most terminal singularities such that $-K_X$ is ample.

\begin{Def}
Let $(X,D)$ be a pair, that is, $D$ is an effective $\mbQ$-divisor, and let $\msp \in X$ be a point.
We define the {\it log canonical threshold} of $(X,D)$ and the {\it log canonical threshold} of $(X,D)$ {\it at} $\msp$ to be the numbers
\[
\begin{split}
\lct (X,D) &= \sup \{\, c \mid \text{$(X, c D)$ is log canonical} \,\}, \\
\lct_{\msp} (X,D) &= \sup \{\, c \mid \text{$(X,D)$ is log canonical at $\msp$} \,\},
\end{split}
\]
respectively.
We define
\[
\lct_{\msp} (X) = \inf \{\, \lct_{\msp} (X,D) \mid D \text{ is an effective $\mbQ$-divisor}, D \equiv -K_X \,\},
\]
and for a subset $\Sigma \subset X$, we define
\[
\lct_{\Sigma} (X) = \inf \{\, \lct_{\msp} (X) \mid \msp \in \Sigma \, \}.
\]
The number $\lct (X) := \lct_X (X)$ is called the {\it global log canonical threshold} or the {\it alpha invariant} of $X$
\end{Def}

Note that the numerical equivalence $D_1 \equiv D_2$ between $\mbQ$-divisors on a Fano variety $X$ is equivalent to $\mbQ$-linear equivalence $D_1 \sim_{\mbQ} D_2$. 
The following relation between $\mult_{\msp} (D)$ and $\lct_{\msp} (X,D)$ which are used to bound $\lct_{\msp} (X,D)$ from below.

\begin{Lem}[{\cite[8.10 Lemma]{Kol}}]
For a nonsingular point $\msp \in X$ and an effective $\mbQ$-divisor $D$ on $X$, the inequality
\[
\frac{1}{\mult_{\msp} (D)} \le \lct_{\msp} (X,D)
\]
holds.
\end{Lem}

We recall the definition of isolating class, introduced by Corti, Pukhlikov and Reid \cite{CPR}, which will play an important role in the computation of multiplicities of divisors at a nonsingular point.

\begin{Def}
Let $L$ be a Weil divisor class on a variety $V$ and $\msp \in V$ a nonsingular point.
We say that $L$ {\it isolates} $\msp$ or it is a $\msp$-{\it isolating class} if $\msp$ is a component of the base locus of $|\mcI_{\msp}^k (k L)|$ for some $k > 0$.
\end{Def}

\subsection{Birationally rigid Fano $3$-folds} \label{sec:birrigcodim2}

Our main objects are birationally rigid Fano $3$-folds of index one embedded in weighted projective spaces as codimension $2$ or $3$ subvarieties, which we simply call {\it codimension $2$ and $3$ Fano $3$-folds}, respectively.
Here, we say that a Fano $3$-fold $X$ is of index one if $-K_X$ is not divisible in the Weil divisor class group.

Codimension $2$ Fano $3$-folds are weighted complete intersections defined by two equations inside $\mbP (a_0,\dots,a_5)$.
They are determined by degrees $d_1, d_2$ of defining polynomials and the weights $a_0,\dots,a_5$ of the ambient space.
Among the $19$ families of birationally rigid codimension $2$ Fano $3$-folds, we consider those which satisfy $(-K_X)^3 \le 1$.
There are $18$ such families, i.e.\ families No.~$i$ with $i \in I^*_{br}$, where
\[
I^*_{br} := \{8,20,24,31,37,45,47,51,59,60,64,71,75,76,84,85\}.
\]
Further details are given in Table \ref{table:codim2Fanos}.
Each family satisfies $d_1 \ne d_2$ and we always assume that $d_1 < d_2$. 
We denote by $F_1$ and $F_2$ the defining polynomials of degree $d_1$ and $d_2$, respectively. 
The singularities of $X$ are described in the 4th column, where $\frac{1}{r} (a,r-a)$, $\frac{1}{2}$ and $\frac{1}{3}$ mean $\frac{1}{r} (1,a,r-a)$, $\frac{1}{2} (1,1,1)$ and $\frac{1}{3} (1,1,2)$, respectively.
For a singular point $\msp \in X$, we denote by $\varphi \colon Y \to X$ the Kawamata blowup which is the unique extremal divisorial extraction of the terminal quotient singular point $\msp$, and we set $B = -K_Y$.
In the 4th column of the table, if the sing $+$ is given as a subscript of the singularity, then it means that $B^3 > 0$ is satisfied.

As is explained in the introduction, there are only $3$ families of codimension $3$ Fano $3$-folds and they are Pfaffian Fano $3$-folds, which are defined in $\mbP (a_0,\dots,a_6)$ by $5$ Pfaffians of a $5 \times 5$ skew symmetric matrix whose entries are homogeneous polynomials.
Detailed descriptions will be given in Section \ref{sec:Pfaff}.

Let $X$ be a codimension 2 or 3 Fano 3-fold and $\mbP = \mbP (a_0,\dots,a_5)$ or $\mbP (a_0,\dots,a_6)$ be the ambient weighted projective space.
We always assume that $a_0 \le a_1 \le \dots$ and the homogeneous coordinates are denoted by $x,y,z,s,t,u$ or $x,y,z,s,t,u,v$.
For a homogeneous coordinate $w$, we denote by $\msp_w$ the point at which only the coordinate $w$ does not vanish.
For example, $\msp_x = (1\!:\!0\!:\!\cdots\!:\!0)$.
For a set of homogeneous coordinates $\{w_1,\dots,w_m\}$, we denote by $\Pi_{w_1,\dots,w_m}$ the quasi-linear subspace of the ambient space $\mbP$ defined by $w_1 = \cdots = w_m = 0$.
For a coordinate $w$, we denote by $H_w$ the hyperplane $X \cap \Pi_w$ of $X$ defined by $w = 0$ on $X$.
We set $L_{xy} = H_x \cap H_y = X \cap \Pi_{x,y}$.
For homogeneous polynomials $g_1, \dots,g_m$, we denote by $(g_1 = \cdots = g_m = 0)$ the subscheme of $\mbP$ defined by the homogeneous ideal $(g_1,\dots,g_m)$ and we set 
\[
(g_1 = \cdots = g_m = 0)_X = (g_1 = \cdots = g_m = 0) \cap X.
\]
We always assume that $X$ is quasi-smooth, which implies that the divisor class group $\Cl (X)$ is isomorphic to $\mbZ$ and it is generated by $-K_X$.
We set $A = -K_X$.

\begin{table}[htb]
\begin{center}
\caption{Codimension $2$ Fano $3$-folds}
\label{table:codim2Fanos}
\begin{tabular}{clcc}
No. & $X_{d_1,d_2} \subset \mbP (a_0,\dots,a_5)$ & $A^3$ & Basket of singularities \\[0.5mm]
\hline \\[-3.5mm]
8 & $X_{4,6} \subset \mbP (1,1,2,2,2,3)$ & $1$ & $6 \times \frac{1}{2}_+$ \\[0.6mm]
14 & $X_{6,6} \subset \mbP (1,2,2,2,3,3)$ & $1/2$ & $9 \times \frac{1}{2}$ \\[0.6mm]
20 & $X_{6,8} \subset \mbP (1,2,2,3,3,4)$ & $1/3$ & $6 \times \frac{1}{2}$, $2 \times \frac{1}{3}_+$ \\[0.6mm]
24 & $X_{6,10} \subset \mbP (1,2,2,3,4,5)$ & $1/4$ & $7 \times \frac{1}{2}$, $\frac{1}{4} (1,3)_+$ \\[0.6mm]
31 & $X_{8,10} \subset \mbP (1,2,3,4,4,5)$ & $1/6$ & $4 \times \frac{1}{2}$, $\frac{1}{3}$, $2 \times \frac{1}{4} (1,3)_+$ \\[0.6mm]
37 & $X_{8,12} \subset \mbP (1,2,3,4,5,6)$ & $2/15$ & $4 \times \frac{1}{2}$, $2 \times \frac{1}{3}$, $\frac{1}{5} (1,4)_+$ \\[0.6mm]
45 & $X_{10,12} \subset \mbP (1,2,4,5,5,6)$ & $1/10$ & $5 \times \frac{1}{2}$, $2 \times \frac{1}{5} (1,4)_+$ \\[0.6mm]
47 & $X_{10,12} \subset \mbP (1,3,4,4,5,6)$ & $1/12$ & $\frac{1}{2}$, $2 \times \frac{1}{3}$, $3 \times \frac{1}{4} (1,3)$ \\[0.6mm]
51 & $X_{10,14} \subset \mbP (1,2,4,5,6,7)$ & $1/12$ & $5 \times \frac{1}{2}$, $\frac{1}{4} (1,3)$, $\frac{1}{6} (1,5)_+$ \\[0.6mm]
59 & $X_{12,14} \subset \mbP (1,4,4,5,6,7)$ & $1/20$ & $2 \times \frac{1}{2}$, $3 \times \frac{1}{4} (1,3)$, $\frac{1}{5} (1,4)$ \\[0.6mm]
60 & $X_{12,14} \subset \mbP (2,3,4,5,6,7)$ & $1/30$ & $7 \times \frac{1}{2}$, $2 \times \frac{1}{3}$, $\frac{1}{5} (2,3)$ \\[0.6mm]
64 & $X_{12,16} \subset \mbP (1,2,5,6,7,8)$ & $2/35$ & $4 \times \frac{1}{2}$, $\frac{1}{5} (2,3)_+$, $\frac{1}{7} (1,6)_+$ \\[0.6mm]
71 & $X_{14,16} \subset \mbP (1,4,5,6,7,8)$ & $1/30$ & $\frac{1}{2}$, $3 \times \frac{1}{4} (1,3)$, $\frac{1}{5} (2,3)$, $\frac{1}{6} (1,5)$ \\[0.6mm]
75 & $X_{14,18} \subset \mbP (1,2,6,7,8,9)$ & $1/24$ & $5 \times \frac{1}{2}$, $\frac{1}{3}$, $\frac{1}{8} (1,7)_+$ \\[0.6mm]
76 & $X_{12,20} \subset \mbP (1,4,5,6,7,10)$ &$1/35$ & $2 \times \frac{1}{2}$, $2 \times \frac{1}{5} (1,4)$, $\frac{1}{7} (3,4)_+$ \\[0.6mm]
78 & $X_{16,18} \subset \mbP (1,4,6,7,8,9)$ & $1/42$ & $2 \times \frac{1}{2}$, $\frac{1}{3}$, $2 \times \frac{1}{4} (1,3)$, $\frac{1}{7} (1,6)$ \\[0.6mm]
84 & $X_{18,30} \subset \mbP (1,6,8,9,10,15)$ & $1/120$ & $2 \times \frac{1}{2}$, $2 \times \frac{1}{3}$, $\frac{1}{5} (1,4)$, $\frac{1}{8} (1,7)$ \\[0.6mm]
85 & $X_{24,30} \subset \mbP (1,8,9,10,12,15)$ & $1/180$ & $\frac{1}{2}$, $\frac{1}{3}$, $\frac{1}{4} (1,3)$, $\frac{1}{5} (2,3)$, $\frac{1}{9} (1,8)$ \\
\end{tabular}
\end{center}
\end{table}

We briefly explain the organizations of this paper.
In the rest of the present section, we explain methods for computing global log canonical threshold of Fano varieties in a relatively general setting.
In Section \ref{sec:comp}, we determine isolating classes of nonsingular points and analyze the properties of specific curves on codimension $2$ Fano $3$-folds, which will play an important role in Section \ref{sec:nonsingpts}.
In Section \ref{sec:nonsingpts} and \ref{sec:singpts}, we compute log canonical thresholds of codimension $2$ Fano $3$-folds at nonsingular points and singular points, respectively.
Finally, we treat codimension 3 Fano $3$-folds in Section \ref{sec:Pfaff} and compute their log canonical thresholds.
Theorem \ref{mainthm} follows from Propositions \ref{prop:lctnspt}, \ref{prop:lctsingpt} for codimension $2$ Fano $3$-folds and from Propositions \ref{prop:lctPfaffnspt}, \ref{prop:lctPfaffsing} for codimension $3$ Fano $3$-folds.

\subsection{Methods for nonsingular points}

In this and next subsection, let $X$ be a Fano $3$-fold with $\Cl (X) \cong \mbZ$ and we assume that $\Cl (X)$ is generated by $A := -K_X$.

We explain two ways to bound multiplicities of divisors at a nonsingular point.

\begin{Lem} \label{lem:exclL}
Let $\msp \in X$ be a nonsingular point.
Suppose that there are distinct prime divisors $S_1 \sim_{\mbQ} c_1 A$ and $S_2 \sim_{\mbQ} c_2 A$ with the following properties.
\begin{enumerate}
\item Both $S_1$ and $S_2$ pass through $\msp$, and $\mult_{\msp} (S_1) \le c_1$.
\item The scheme theoretic intersection $\Gamma := S_1 \cap S_2$ is an irreducible and reduced curve such that $\mult_{\msp} (\Gamma) \le c_2$.
\item The inequality $c_1 c_2 A^3 \le 1$ holds.
\end{enumerate}
Then $\mult_{\msp} (D) \le 1$ for any effective $\mbQ$-divisor $D \sim_{\mbQ} A$.
\end{Lem}

\begin{proof}
We assume that the conclusion does not hold.
Then there is an irreducible $\mbQ$-divisor $D \sim_{\mbQ} A$ such that $\mult_{\msp} (D) > 1$.
By (1), we have $\mult_{\msp} (\frac{1}{c_1} S_1) \le 1$ and this implies $S_1 \ne \Supp (D)$.
Then we can write $D \cdot S_1 = \gamma \Gamma + \Delta$, where $\gamma \ge 0$ and $\Delta$ is an effective $1$-cycle such that $\Gamma \not\subset \Supp (\Delta)$.
Since $A$ is ample, we have 
\[
c_1 A^3 = A \cdot D \cdot S_1 \ge \gamma A \cdot \Gamma = \gamma c_1 c_2 A^3,
\]
which implies $\gamma \le 1/c_2$. 
We set $d = A^3$ and $m = \mult_{\msp} (L)$.
Since $S_1 \cap S_2 = \Gamma$, any component of $\Delta$ is not contained in $S_2$ and we have
\[
c_1 c_2 d - \gamma c_1 c_2^2 d = S_2 \cdot (D \cdot S_1 - \gamma \Gamma) = S_2 \cdot \Delta \ge \mult_{\msp} (\Delta) > 1 - \gamma m
\]
and thus
\[
(m - c_1 c_2^2 d) \gamma > 1 - c_1 c_2 d.
\]
By (3), we have $1 - c_1 c_2 d \ge 0$, which implies $m - c_1 c_2^2 d > 0$.
Combining $\gamma \le 1/c_2$ and the last displayed inequality, we have $m > c_2$.
This is a contradiction since $m = \mult_{\msp} (\Gamma) \le c_2$ by (2).
\end{proof}

\begin{Lem} \label{lem:exclG}
Let $\msp \in X$ be a nonsingular point.
Suppose that one of the following conditions is satisfied.
\begin{enumerate}
\item 
There is a $\msp$-isolating class $lA$ and distinct prime divisors $S_1 \sim_{\mbQ} c_1 A$, $S_2 \sim_{\mbQ} c_2 A$ such that $\max \{c_1,c_2\} l A^3 \le 1$. 
\item There is a $\msp$-isolating class $l A$ and a prime divisor $S \sim_{\mbQ} c A$ such that $\mult_{\msp} (S) \le c$ and $c l A^3 \le 1$.
\end{enumerate}
Then $\mult_{\msp} (D) \le 1$ for any effective $\mbQ$-divisor $D \sim_{\mbQ} A$.
\end{Lem}

\begin{proof}
Assume that the conclusion does not hold.
Then there is an irreducible $\mbQ$-divisor $D \sim_{\mbQ} A$ such that $\mult_{\msp} (D) > 1$.

If the condition (1) is satisfied, then we may assume $\Supp (D) \ne S_1$ after possibly interchanging $S_1,S_2$ and we set $S = S_1$, $c = c_1$.
If the condition (2) is satisfied, then $\Supp (D) \ne S$ since $\mult_{\msp} (S) \le c$.
In any case, $D \cdot S$ is an effective $1$-cycle and $c l A^3 \le 1$.
Since $lA$ isolates $\msp$, $\Bs |\mcI_{\msp}^k (klA)|$ is the union of finitely many points and curves $\Gamma_1,\dots,\Gamma_m$ which do not pass through $\msp$ for $k \gg 0$.
We write $D \cdot S = \Gamma + \Delta$, where $\Gamma, \Delta$ are effective $1$-cycle, $\Gamma$ is supported on $\Gamma_1 \cup \cdots \cup \Gamma_m$ and $\Gamma_1,\dots,\Gamma_m \not\subset \Supp (\Delta)$. 
We have $\mult_{\msp} (\Delta) = \mult_{\msp} (D \cdot S) > 1$.
Since $T$ is nef, we have
\[
k c l A^3 = D \cdot S \cdot T = T \cdot \Gamma + T \cdot \Delta \ge T \cdot \Delta \ge k \mult_{\msp} (\Delta) > k,
\]
which implies $c l A^3 > 1$.
This is a contradiction. 
\end{proof}

In some places we use the following result in order to obtain a divisor vanishing at a given point at least doubly.

\begin{Lem} \label{lem:highmultdiv}
Let $V$ be a normal projective variety embedded in a weighted projective space $\mbP = \mbP (1,a_1\dots,a_n)$ with homogeneous coordinates $x_0,\dots,x_n$, and let $\msp \in V$ be a nonsingular point which is not contained in $H_{x_0} = (x_0 = 0) \cap V$.
Then there are $1 \le i_1 < i_2 < \cdots < i_c \le n$ such that $|\mcI_{\msp}^2 (a_{i_k} A)| \ne \emptyset$, where $c$ is the codimension of $V \subset \mbP$.
\end{Lem}

\begin{proof}
Replacing coordinates, we may assume $\msp = (1\!:\!0\!:\!\cdots\!:\!0)$ and we work on the open subset $\mbA^n_{x_1,\dots,x_n} \subset \mbP$ on which $x_0 \ne 0$. 
Since $V$ is nonsingular at $\msp$, it is defined by $c$ equations $f_1 = f_2 = \cdots = f_c = 0$ locally around an open subset $U \subset \mbA^n$. 
We will freely shrink $U \ni \msp$.
By setting $x_0 = 1$, the equations are of the form
\[
f_i = \sum_{j=1}^n \alpha_{i j } x_j + \text{higher terms}.
\]
By a linear change of the $f_i$ and a coordinate change, we may assume that
\[
f_1 = x_{i_1} + g_1, f_2 = x_{i_2} + g_2, \cdots, f_c = x_{i_c} + g_c,
\]
where $1 \le i_1 < i_2 < \cdots < i_c \le n$ and $g_i \in (x_1,\dots,x_n)^2$. 
Here, by a coordinates change, we mean the combination of coordinate changes of the form $x_j \mapsto \sum_{l \le j} \beta_l x_l$ with $\beta_j \ne 0$.
These coordinate changes can be realized as the restriction of coordinate changes of $\mbP$.
Now the assertion follows immediately since $H_{x_{i_k}} \in |a_{i_k} A|$ is singular at $\msp$ for $k = 1,2,\dots,c$. 
\end{proof}

\subsection{Methods for singular points}

Let $\msp \in X$ be a terminal quotient singular point of type $\frac{1}{r} (1,a,r-a)$, where $r > 1$ and $0 < a < r$ is coprime to $r$.
We denote by $\varphi \colon Y \to X$ the Kawamata blowup at $\msp$ and by $E \cong \mbP (1,a,r-a)$ the exceptional divisor of $\varphi$.
We set $B = -K_Y$.
For a curve $\Gamma$ or a divisor $D$ on $X$, we denote by $\tilde{\Gamma}$ and $\tilde{D}$ their proper transforms via $\varphi$. 

\begin{Lem} \label{lem:singptNE}
Suppose that $B^2 \notin \Int \bNE (Y)$ and there exists a prime divisor $S$ on $X$ such that $\tilde{S} \sim_{\mbQ} m B$ for some $m > 0$.
Then $\lct_{\msp} (X) \ge 1$.
\end{Lem}

\begin{proof}
Assume that the conclusion does not hold.
Then there is an irreducible $\mbQ$-divisor $D \sim_{\mbQ} A$ on $X$ such that $(X,D)$ is not log canonical at $\msp$.
We write $\varphi^*D = \tilde{D} + \mu E$, where $\mu \in \mbQ$.
Then we have $\mu > 1/r$ by \cite{Kawamata}.
In particular $\tilde{D} \ne \tilde{S}$ and we have
\[
\bNE (Y) \ni \tilde{S} \cdot \tilde{D} = m B \cdot \left(B + \left( \frac{1}{r}-\mu \right) E \right).
\]
It follows that
\[
m B^2 = \tilde{S} \cdot \tilde{D} + m \left(\mu - \frac{1}{r} \right) B \cdot E \in \Int \bNE (Y)
\]
since $B \cdot E$ generates the extremal ray $R \subset \bNE (Y)$ which defines $\varphi$ and clearly $\tilde{S} \cdot \tilde{D} \notin R$.
This is a contradiction and the proof is completed.
\end{proof}

\begin{Lem} \label{lem:singptnef}
Let $N = a \varphi^*A - \frac{e}{r} E$ be a nef divisor on $Y$.
Suppose that there are distinct prime divisors $S_1, S_2$ on $X$ such that 
\[
r^3 a a_i (A^3) \le e e_i (E^3),
\]
where $a_i, e_i$ are positive integers defined as $\tilde{S}_i \sim_{\mbQ} a_i \varphi^*A - \frac{e_i}{r} E$ for $i = 1,2$.
Then $\lct_{\msp} (X) \ge 1$.
\end{Lem}

\begin{proof}
Assume that the conclusion does not hold.
Then there is an irreducible $\mbQ$-divisor $D \sim_{\mbQ} A$ such that $(X,D)$ is not log canonical at $\msp$.
Then we have $\mu > 1/r$, where $\tilde{D} \sim_{\mbQ} \varphi^*A - \mu E$.
We may assume $D \ne S_1$ after interchanging $S_1$ and $S_2$.
Then $\tilde{D} \cdot \tilde{S}_1$ is an effective $1$-cycle on $Y$ and we have $N \cdot \tilde{D} \cdot \tilde{S}_1 \ge 0$.
We compute
\[
\begin{split}
N \cdot \tilde{D} \cdot \tilde{S}_1 &= (a \varphi^*A - \frac{e}{r} E) \cdot  (\varphi^*A - \mu E) \cdot (a_1 \varphi^*A - \frac{e_1}{r} E) \\
&= a a_1 (A^3) - \mu \frac{e e_1}{r^2} (E^3) \\
& < a a_1 (A^3) - \frac{e e_1}{r^3} (E^3) \le 0.
\end{split}
\]
This is a contradiction and the assertion is proved.
\end{proof}

For a terminal quotient singular point $\msq$ of index $r$ on a variety $V$, we denote by $\rho_{\msq} \colon \breve{V}_{\msq} \to V$ the index $1$ cover of an open neighborhood of $\msq \in V$, that is, $\breve{V}_{\msq}$ is nonsingular and $\mbZ/r \mbZ$ acts on it and the quotient is an open subset of $V$.
We denote by $\breve{\msq} \in \breve{V}_{\msq}$ the preimage of the point $\msq$.
For a divisor $D$ on $V$, a birational morphism $\psi \colon W \to V$ and an irreducible exceptional divisor $G$ of $\psi$, we denote by $\ord_G (D)$ the rational number which is the coefficient of $G$ in $\psi^*D$.

\begin{Lem} \label{lem:singptB3neg}
Let $V$ be a normal projective $\mbQ$-factorial $3$-fold such that $-K_V$ is nef and big, and let $\msq \in V$ be a terminal quotient singular point of index $r$.
Suppose that there are prime divisors $S_1, S_2$ on $V$ with the following properties.
\begin{enumerate}
\item $S_1 \sim_{\mbQ} -c_1 K_V$ and $S_2 \sim_{\mbQ} - c_2 K_V$ for some positive $c_1,c_2 \in \mbQ$.
\item $\ord_F (S_1) \le c_1/r$, where $F$ is the exceptional divisor of the Kawamata blowup at $\msq \in V$.
\item The scheme-theoretic intersection $\breve{\Gamma} := \rho_{\msq}^*S_1 \cap \rho_{\msq}^*S_2$ is an irreducible and reduced curve such that $\mult_{\breve{\msq}} (\breve{\Gamma}) \le c_2$.
\item $r c_1 c_2 (-K_V)^3 \le 1$.
\end{enumerate}
Then $(V,D)$ is log canonical at $\msp$ for any effective $\mbQ$-divisor $D \sim_{\mbQ} -K_V$.
\end{Lem}

\begin{proof}
We write $\rho = \rho_{\msq}$ and $\breve{V} = \breve{V}_{\msq}$.
Assume that the conclusion does not hold.
Then there is an irreducible $\mbQ$-divisor $D \sim_{\mbQ} -K_V$ such that $(V,D)$ is not log canonical at $\msq$.
Then $\ord_F (D) > 1/r$, and in particular $\Supp D \ne S_1$ by (2).
Moreover, $(\breve{V}, \rho^*D)$ is not log canonical at $\breve{\msq}$ and we have $\mult_{\breve{\msq}} (\rho^* D) > 1$.
We write $\rho^*D \cdot \rho^*S_1 = \gamma \breve{\Gamma} + \breve{\Delta}$, where $\breve{\Delta}$ is an effective $1$-cycle on $\breve{V}$ such that $\breve{\Gamma} \not\subset \breve{\Delta}$, and write $D \cdot S_1 = \gamma \Gamma + \Delta + \Xi$, where $\Gamma = S_1 \cap S_2$ is an irreducible and reduced curve, $\Delta = \frac{1}{r} \rho_* \breve{\Delta}$ and $\Xi$ is an effective $1$-cycle such that $\Gamma \not\subset \Supp (\Xi)$.
The $1$-cycle $\Xi$ may appear since we only consider the index $1$ cover of an open neighborhood of $\msq \in V$.
We set $d = (-K_V)^3$ which is positive since $-K_V$ is nef and big.
Since $-K_V$ is nef, we have 
\[
c_1 d = (-K_V) \cdot D \cdot S_1 \ge \gamma (-K_V) \cdot L = \gamma (-K_V) \cdot S_1 \cdot S_2 = \gamma c_1 c_2 d,
\]
which implies $\gamma \le 1/c_1$.
We have 
\[
\begin{split}
r (c_1 c_2 d - \gamma c_1 c_2^2 d) &= r (D \cdot S_1 \cdot S_2 - \gamma S_2 \cdot L) =  r (S_2 \cdot (\Delta + \Xi)) \\
& \ge r (S_2 \cdot (\Delta + \Xi))_{\msp} \ge r (S_2 \cdot \Delta)_{\msp} =  (\rho^*S_2 \cdot \breve{\Delta})_{\breve{\msp}} \\
& > 1 - m \gamma,
\end{split}
\]
where $m = \mult_{\breve{\msp}} (\breve{L})$ and $(\ , \ )_{\msp}$ denotes local intersection number.
It follows that
\[
(m - c_1 c_2^2 d) \gamma > 1 - r c_1 c_2 d.
\]
This in particular implies $m - c_1 c_2^2 d > 0$ since $1 - r c_1 c_2 d \ge 0$ by the condition (4).
By combining the last displayed inequality and $\gamma \le 1/c_2$, we have
\[
\frac{1}{c_2} (m - c_1 c_2^2 d) > 1 - r c_1 c_2 d,
\]
which implies $m > c_2$.
This is a contradiction.
\end{proof}

\subsection{Sarkisov involutions of flopping type and log canonical thresholds}
\label{sec:Sinvolandlct}

Let $X$ be a $\mbQ$-Fano $3$-fold such that $\Cl (X) \cong \mbZ$ and $A = -K_X$ is the generator of $\Cl (X)$.
For an effective (integral) Weil divisor $D$ on $X$, we define $n_D \in \mbZ$ be the non-negative integer such that $D \in |n_D A|$.

Let $\msp \in X$ be a terminal quotient singular point of type $\frac{1}{r} (1,a,r-a)$ (i.e.\ $r > 1$, and $a$ and $r$ are coprime) and let $\varphi \colon Y \to X$ be the Kawamata blowup at $\msp$ with exceptional divisor $E$.
We assume that there is a diagram
\[
\xymatrix{
Y \ar[d]_{\varphi} \ar@{-->}[r]^\tau & Y \ar[d]^{\varphi} \\
X & X}
\]
where $\tau$ is a flop, and that the induced map $\sigma \colon X \ratmap X$ is not biregular.
Let $\psi \colon Y \to Z$ be the flopping contraction and $\pi \colon X \ratmap Z$ be the induced birational map.
We denote by $\Exc (\pi)$ and $\Exc (\psi)$ the exceptional loci of $\pi$ and $\psi$, respectively.
Note that $\Exc (\psi)$ is the proper transform of $\Exc (\pi)$ via $\varphi$.
We see that $\tau_* E$ is a prime divisor on $Y$ such that $\tau_* E \ne E$.
It follows that $\tau_* E = \tilde{G}$, where $G = \varphi_*\tau_*E$ is a prime divisor on $X$.

It is easy to see that any effective divisor on $Y$ is $\mbQ$-linear equivalent to $\alpha B + \beta E$ for some $\alpha \ge 0$, and the cone $\Eff (Y)$ of effective divisors on $Y$ is generated by $E$ and $\tilde{G}$.
For a divisor $D$ on $X$, we define $\mu_D = \ord_E (D)$ and $e_D = \mu_D/n_D - 1/r$.
We have $\tilde{D} \sim_{\mbQ} n_D (B - e_D E)$.

\begin{Lem} \label{lem:birinvbasics}
The following assertions hold for a prime divisor $D$ on $X$.
\begin{enumerate}
\item If $e_D \le 0$, then $(X, \frac{1}{n} D)$ is log canonical at $\msp$, and in particular, the pair $(Y, \frac{1}{n_D} \tilde{D} + e_D E)$ is sub log canonical at any point of $E$.
\item We have $e_D \le e_G = 1/n_G \le 1$.
\end{enumerate}
\end{Lem}

\begin{proof}
(1) follows from the result of \cite{Kawamata} and the equation
\[
K_Y + \frac{1}{n_D} \tilde{D} + e_D E = \varphi^* \left( K_X + \frac{1}{n_D} D \right).
\]

Since $\tau_*$ is a involution of $\Cl (Y)$, $\tau_*B = B$ and $\tau_* E = \tilde{G}$, we have $\tilde{G} = \tau_* E \sim_{\mbQ} \alpha B - E$ for some $\alpha \in \mbZ$.
We have $\alpha > 0$ since $G$ is effective.
On the other hand, we have $\tilde{G} \sim_{\mbQ} n_G (B - e_G E)$.
Thus $n_G = \alpha$ and $e_G = 1/n_G$.
Now since $\tilde{D} \sim_{\mbQ} n B - e E \in \Eff (Y)$ and $\Eff (Y)$ is generated by $E$ and $\tilde{G}$, we have $e \le e_G = 1/n_G \le 1$.
This completes the proof.
\end{proof}

\begin{Lem}
Suppose that $\operatorname{lct}_{G \setminus \{\msp\}} (X) \ge 1$ and that $(Y, \frac{1}{n_G} (\tilde{G} + E))$ is log canonical at any point of $\tilde{G} \cap E$.
Then $\lct_{\msp} (X) \ge 1$.
\end{Lem}

\begin{proof}
Assume that the conclusion does not hold.
Then there is a prime divisor $D$ on $X$ such that $(X, \frac{1}{n_D} D)$ is not log canonical at $\msp$.
We set $n = n_D$, $\mu = \mu_D$ and $e = e_D = \mu/n - 1/r$.
By Lemma \ref{lem:birinvbasics}, we have $0 < e \le 1$.
Set $\Delta = \frac{1}{n} \tilde{D} + e E$.
We have 
\[
K_Y + \Delta =  K_Y + \frac{1}{n} \tilde{D} + \left( \frac{\mu}{n} - \frac{1}{r} \right) E = \varphi^* \left( K_X + \frac{1}{n} D \right) \sim_{\mbQ} 0
\]
It follows that $\Delta \sim_{\mbQ} B$ and the pair $(Y,\Delta)$ is not log canonical at some point of $E$.
Since $\tau$ is a flop and $K_Y + \Delta$ is $\psi$-trivial, we see that $(Y,\tau_*\Delta)$ is not sub log canonical at some point of $\tau_* E = \tilde{G}$ and we have $\tau_* \Delta \sim_{\mbQ} B$.
We have $\tau_*\Delta = \frac{1}{n} \tau_*\tilde{D} + e \tilde{G}$.
Set $D' = \varphi_*\tau_*\tilde{D}$.
We see that $\tau_*\tilde{D}$ is the proper transform of $D'$ via $\varphi$.
Since
\[
\tilde{D}' = \tau_* \tilde{D} \sim_{\mbQ} n B - n e (n_G B - E) = n (1-n_G e) B + n e E,
\]
we have $n' := n_{D'} = n(1 - n_G e)$.
Note that $0 < n_G e < 1$ since $n' > 0$.
By setting $\alpha = n_G e$, we can write
\[
\tau_*\Delta = (1-\alpha) \left(\frac{1}{n'} \left( \tilde{D}' + e' E \right) \right) + \alpha \left( \frac{1}{n_G} \left(\tilde{G} + E \right) \right).
\]
It follows that either $(Y, \frac{1}{n'} (\tilde{D}' + e'E))$ or $(Y, \frac{1}{n_G} (\tilde{G} + E))$ is not (sub) log canonical at a point of $\tilde{G}$.
By the assumption $\lct_{G \setminus \{\msp\}} (X) \ge 1$ and Lemma \ref{lem:birinvbasics}, the pair $(Y, \frac{1}{n'} (\tilde{D}' + e'E))$ is sub log canonical at any point of $\tilde{G}$.
Also, by the assumption, the pair $(Y, \frac{1}{n_G} (\tilde{G} + E))$ is log canonical at any point of $\tilde{G}$.
This is a contradiction and the proof is completed.
\end{proof}

\begin{Lem}
Suppose that $n_G \ge 2$. 
Then the pair $(Y, \frac{1}{n_G} (\tilde{G} + E))$ is log canonical at every point of $Y \setminus \Exc (\psi)$.
\end{Lem}

\begin{proof}
We claim that $(Y, E)$ is log canonical.
Indeed, $(Y,E)$ is clearly log canonical at any nonsingular point of $Y$.
Let $\msq$ be a singular point of $Y$ which is contained in $E$.
Then, locally around $\msq$, the pair $(Y,E)$ is isomorphic to the quotient of a pair $(\breve{Y}, \breve{E})$ by a suitable cyclic group, where both $\breve{Y}$ and $\breve{E}$ are nonsingular at the preimage $\breve{\msq}$ of $\msq$.
It follows that $(\breve{Y},\breve{E})$ is log canonical at $\breve{\msq}$ and thus $(Y,E)$ is log canonical at $\msq$.

The birational involution $\tau$ induces a biregular involution of $Y \setminus \Exc (\psi)$ which maps $E \setminus \Exc (\psi)$ isomorphically onto $\tilde{G} \setminus \Exc (\psi)$.
It follows that $(Y, \tilde{G})$ is log canonical at any point of $Y \setminus \Exc (\psi)$.
This implies that $(Y,\frac{1}{2} (\tilde{G} + E))$ is log canonical at every point of $Y \setminus \Exc (\psi)$, hence so is $(Y,\frac{1}{n_G} (\tilde{G} + E))$ since $n_G \ge 2$.
\end{proof}

Combining the above results, we have the following.

\begin{Prop} \label{prop:lctselflink}
Suppose that $n_G \ge 2$, $\lct_{G \setminus \{\msp\}} (X) \ge 1$ and $(Y, \frac{1}{n_G} (\tilde{G} + E))$ is log canonical at any point of $E \cap \Exc (\psi)$, then $\lct_{\msp} (X) \ge 1$.
\end{Prop}

\subsection{Generality assumptions}

For a member $X$ of family No.~$i \in I^*_{br}$, we introduce the following conditions.

\begin{Cond} \label{cd}
\begin{enumerate}
\item $X$ is quasi-smooth.
\item The conditions given in \cite{Okada1} are satisfied.
\item If $i \ne 60$, then the anticanonical linear system $\left|-K_X\right|$ contains a quasi-smooth member.
\end{enumerate}
\end{Cond}

It is clear that the above conditions are satisfied for general members of family No.~$i$.
We introduce further conditions on specific families.

\begin{Cond} \label{cdsp}
\begin{enumerate}
\item If $i \in \{45,51,64,75\}$, then the conclusion of Lemma \ref{lem:gennsptmult2} holds. 
\item If $i = 8$, then the quadratic form $F_1 (0,0,z,s,t,0)$ is of rank $3$ and the conclusions of Lemma \ref{lem:No8nonsing1} and \ref{lem:No8spnsptgen} hold.
\item If $i \in \{8,20,24,31,37\}$, then the conclusion of Lemma \ref{lem:QIexcgen} holds.
\end{enumerate}
\end{Cond}

The first part of \ref{cdsp}.(2) is clearly satisfied for a general members and, for the other conditions, detailed arguments will be given in the corresponding lemma.

\section{Various computations} \label{sec:comp}

We will compute global log canonical thresholds of codimension 2 Fano $3$-folds in Sections \ref{sec:nonsingpts} and \ref{sec:singpts} by mainly applying the methods given in Section \ref{sec:prelim}.
When we apply those methods it is required to understand the singularity of $L_{xy} = H_x \cap H_y$ and $\msp$-isolating classes for nonsingular points.

\subsection{Singularities of $L_{xy}$}

In this subsection, we prove that, in most of the cases, $L_{xy} = H_x \cap H_y$ is an irreducible and reduced curve and $\Sing (L_{xy}) \subset \Sing (X)$ for $X \subset \mbP (a_0,\dots,a_5)$ such that $1 = a_0 \le a_1 < a_2$.
There are $12$ such families.
Note that divisors $H_x$ and $H_y$ depends on the choice of coordinates while their intersection $L_{xy}$ does not since $a_1 < a_2$.

\begin{Prop} \label{prop:irredLxy}
Let $X$ be a member of family No.~$i$.
\begin{enumerate}
\item Suppose that 
\[
i \in \{31,37,45,47,51,64,71,75,76,78,84,85\},
\]
and that the condition indicated in the $4$th column of \emph{Table \ref{table:Lxy}} is satisfied for $i \in \{45,51\}$.
Then $L_{xy}$ is an irreducible and reduced curve and $\Sing (L_{xy}) \subset \Sing (X)$.
\item Suppose that $i = 8$, then $L_{xy}$ is an irreducible nonsingular curve of degree $1$.
\end{enumerate}
\end{Prop}

\begin{proof}
We prove (1).
We write $F_1 = G_1 + H_1$ and $F_2 = G_2 + H_2$, where $G_j \in \mbC [z,s,t,u]$ and $H_j \in (x,y) \subset \mbC [x,y,\dots,u]$.
Then $L_{xy}$ is isomorphic to the closed subscheme defined by $G_1 = G_2 = 0$ in $\mbP (a_2,a_3,a_4,a_5)$.
Quasi-smoothness of $X$ implies the presence of some monomials in $G_j$ and after re-scaling coordinates, the equations $G_1 = G_2 = 0$ can be transformed into the form given the second column of Table \ref{table:Lxy} (see Example \ref{ex:Lxy} below).
It is easy to see that $\Sing (L_{xy})$ is contained in the set described in the third column of Table \ref{table:Lxy}.

We prove (2).
We can write 
\[
F_1 = q (z,s,t) + u f_1 + G_1, \ F_2 = c (z,s,t) + u^2 + G_2,
\] 
where $q, c$ are quadratic, cubic forms in $z,s,t$, respectively, $G_j = G_j (x,y,z,s,t) \in (x,y,u)^2$ and $f_1 \in \mbC [x,y]$.
Then $L := L_{xy} \cong (q = c + u^2 = 0) \subset \mbP (2,2,2,3)$ and the Jacobian matrix of the affine cone $C_{L}$ of $L$ is
\[
J_{C_L} =
\begin{pmatrix}
\frac{\prt q}{\prt z} & \frac{\prt q}{\prt s} & \frac{\prt q}{\prt t} & 0 \\[2mm]
\frac{\prt c}{\prt z} & \frac{\prt c}{\prt s} & \frac{\prt c}{\prt t} & 2 u
\end{pmatrix}.
\]
The first row is of rank $1$ at any point of $L$ since $q$ is of rank $3$, which implies that $\Sing (L)$ is contained in the locus $(u = 0)$.
But then $J_{C_L}|_{(u=0)}$ is of rank $2$ since $X$ is quasismooth at any point of $L_{xy} \cap (u=0)$.
Therefore (2) is proved. 
\end{proof}

\begin{Ex} \label{ex:Lxy}
Let $X = X_{8,10} \subset \mbP (1,2,3,4,4,5)$ be a member of family No.~$31$.
We have $u^2 \in F_2$ and we can write
\[
G_1 = \alpha u z + q (s,t), \ 
G_2 = u^2 + z^2 (\lambda s + \mu t),
\]
for some $\alpha, \lambda, \mu \in \mbC$ and a quadratic form $q (s,t)$.
The equation $q (s,t) = 0$ has distinct solutions and we may assume $q (s,t) = st$.
We have $\alpha \ne 0$ by the condition $(\mathrm{C}_3)$ and we may assume $\alpha = 1$ by re-scaling $z$.
If $\lambda = 0$, then $X$ contains the curve $(x = y = t = u = 0)$, which is impossible by the condition $(\mathrm{C}_1)$.
By the same reason, we have $\mu \ne 0$ and we obtain the desired form.
Note that $L_{xy}$ is irreducible and reduced, and $\Sing (L_{xy}) = \{\msp_s,\msp_t\}$.

Let $X = X_{8,12} \subset \mbP (1,2,3,4,5,6)$ be a member of family No.~$37$.
By $(\mathrm{C}_1)$, we have $s^2 \in F_1$ and we can write
\[
G_1 = t z + s^2, \ 
G_2 = u^2 + \alpha u z^2 + \beta t s z + \lambda s^3 + \gamma z^4,
\]
for some $\alpha,\beta,\gamma,\lambda \in \mbC$.
Replacing $u$, we may assume $\alpha = 0$.
Replacing $F_2$ with $F_2 - \beta s F_1$, we may assume $\beta = 0$.
Then, we see that $\gamma \ne 0$ because otherwise $X$ is not quasi-smooth at $\msp_z \in X$.
Re-scaling coordinates, we may assume $\gamma = 1$.
If $\lambda = 0$, then $X$ contains the WCI curve $(x = y = u - \sqrt{-1} z^2 = 0)$ and this is impossible by $(\mathrm{C}_1)$.
Thus $\lambda \ne 0$ and we obtain the desired form.

Let $X = X_{10,12} \subset \mbP (1,3,4,4,5,7)$ be a member of family No.~$47$.
We have $t^2 \in F_1$ and $u^2 \in F_2$, hence we can write
\[
G_1 = \alpha u z + \beta u s + t^2, \ 
G_2 = u^2 + c (z,s),
\]
for some $\alpha, \beta \in \mbC$ and a cubic form $c (z,s)$.
Since $c (z,s) = 0$ has distinct solutions, we can arrange $z,s$ so that $c (z,s) = z s (\lambda z + \mu s)$, where $\lambda, \mu \ne 0$. 
We see $\msp_z, \msp_s \in X$, hence $u z, u s \in F_1$ by quasi-smoothness at $\msp_z,\msp_s$. 
Now, by re-scaling coordinates, we can assume $\alpha = \beta = 1$ and we obtain the desired form.
We see that $L_{xy}$ is an irreducible and reduced nonsingular curve for any choice of non-zero $\lambda, \mu \in \mbC$.

Let $X = X_{10,14} \subset \mbP (1,2,4,5,6,7)$ be a member of family No.~$51$.
We have $t z, s^2 \in F_1$ and $u^2 \in F_2$ and we can write
\[
G_1 = t z + s^2, \ 
G_2 = u^2 + \alpha t z^2 + \beta s^2 z,
\]
for some $\alpha,\beta \in \mbC$.
Replacing $F_2$ with $F_2 - \beta z F_1$, we can eliminate $s^2 z$ and we obtained the desired form.
$L_{xy}$ is always irreducible, and it is reduced if and only if $\alpha \ne 0$.

Let $X = X_{12,16} \subset \mbP (1,2,5,6,7,8)$ be a member of family No.~$64$.
We have $t z, s^2 \in F_1$ and $u^2 \in F_2$, and we can write
\[
G_1 = t z + s^2, \  
G_2 = u^2 + \lambda s z^2,
\]
for some $\lambda \in \mbC$.
If $\lambda = 0$, then $H_x$ is not quasi-smooth at $\msp_z$.
This is impossible by the generality condition and we have $\lambda \ne 0$.
Thus we have the desired form. 
\end{Ex}

\begin{table}[htb]
\begin{center}
\caption{Equations and singularities of $L_{xy}$}
\label{table:Lxy}
\begin{tabular}{cccc}
No. & Equations $G_1 = G_2 = 0$ & $\Sing (L_{xy})$ & Cond \\
\hline
31 & $u z + s t = u^2 + \lambda s z^2 + \mu t z^2 = 0$, $\lambda \mu \ne 0$ & $\{\msp_s,\msp_t\}$ & none \\
37 & $t z + s^2 = u^2 + \lambda s^3 + z^4 = 0$, $\lambda \ne 0$ & $\emptyset$ & none \\
45 & $\alpha u z + s t = u^2 + z^3 = 0$ & $\{\msp_s,\msp_t\}$ & $\alpha \ne 0$ \\
47 & $u z + u s + t^2 = u^2 + z s (\lambda z + \mu s) = 0, \lambda \mu \ne 0$ & $\emptyset$ & none \\
51 & $t z + s^2 = u^2 + \alpha t z^2 = 0$ & $\{\msp_t\}$ & $\alpha \ne 0$ \\
64 & $t z + s^2 = u^2 + \lambda s z^2 = 0, \lambda \ne 0$ & $\{\msp_t\}$ & none \\
71 & $u s + t^2 = u^2 + \lambda s z^2 = 0, \lambda \ne 0$ & $\{\msp_z,\msp_s\}$ & none \\
75 & $t z + s^2 = u^2 + z^3 = 0$ & $\{\msp_t\}$ & none \\
76 & $t z + s^2 = u^2 + \lambda t^2 s + z^4 = 0$, $\lambda \ne 0$ & $\emptyset$ & none \\
78 & $u s + t^2 = u^2 + z^3 = 0$ & $\{\msp_s\}$ & none \\
84 & $t z + s^2 = u^2 + t^3 = 0$ & $\{\msp_z\}$ & none \\
85 & $\lambda u z + t^2 = u^2 + t z^2 + s^3 = 0, \lambda \ne 0$ & $\emptyset$ & none \\
\end{tabular}
\end{center}
\end{table}

\begin{Prop} \label{prop:Lxysp}
Let $X$ be a member of family No,~$i$, where $i \in \{45,51\}$, and suppose that the condition indicated in the $4$th column of \emph{Table \ref{table:Lxy}} is not satisfied.
Then the following hold.
\begin{enumerate}
\item If $i = 45$, then $L_{xy} = \Gamma_1 + \Gamma_2$, where $\Gamma_1, \Gamma_2$ are irreducible and reduced nonsingular curves such that $A \cdot \Gamma_1 = A \cdot \Gamma_2 = 1/10$ and they intersect only at one singular point of type $\frac{1}{2} (1,1,1)$.
\item If $i = 51$, then $L_{xy} = 2 \Gamma$, where $\Gamma$ is an irreducible and reduced nonsingular curve such that $A \cdot \Gamma = 1/12$.
\end{enumerate}
\end{Prop}

\begin{proof}
This follows immediately by setting $\alpha = 0$ in the equations in Table \ref{table:Lxy}.
\end{proof}

\subsection{Isolating classes for nonsingular points}

In this subsection, we seek for isolating classes of nonsingular points of $X$.
This is already studied in \cite{Okada1}, but we need sharper estimates for our purpose.

Let $V$ be a normal projective variety embedded in a weighted projective space $\mbP = \mbP (a_0,\dots,a_n)$ with homogeneous coordinates $x_0,\dots,x_n$, and let $A$ be a Weil divisor on $V$ such that $\mcO_V (A) \cong \mcO_V (1)$.

\begin{Def}
Let $\msp \in V$ a nonsingular point.
We say that a finite set $\{g_i\}$ of homogeneous polynomials {\it isolates} $\msp$ if $\msp$ is a component of 
\[
V \cap \bigcap_i (g_i = 0).
\]
\end{Def}

\begin{Lem}
If a finite set $\{g_i\}$ of homogeneous polynomials isolates $\msp$, then $l A$ isolates $\msp$, where $l = \max \{ \deg g_i \}$.
\end{Lem}

\begin{Lem} \label{lem:findisol}
Let $\pi \colon V \ratmap \mbP (a_0,\dots,a_m)$ be the projection by the coordinates $x_0,\dots,x_m$ and suppose that $\pi$ is a finite morphism.
If $\msp \notin H_{x_j}$, where $0 \le j \le m$, then $lA$ isolates $\msp$, where
\[
l = \max_{0 \le k \le m, k \ne j} \{ \operatorname{lcm} (a_j,a_k)\}.
\]
\end{Lem}

\begin{proof}
We set $\msp = (\xi_0\!:\!\cdots\!:\!\xi_n)$.
Then $\pi (\msp) = (\xi_0\!:\!\cdots\!:\!\xi_m)$ and $\xi_j \ne 0$.
For $k \ne j$, we put $a'_k = a_k/\gcd (a_j,a_k)$ and $a'_j = a_j/\gcd (a_j,a_k)$.
It is easy to see that the common zero loci of the sections in
\[
\{g_0,\dots,\hat{g}_j,\dots,g_m\}, \text{where $g_k = \xi_j^{a'_k} x_k^{a'_j} - \xi_k^{a'_j} x_j^{a'_k}$},
\]
is $\{\pi (\msp)\}$.
It follows that the common zero loci of the above set, considered as the section on $V$, is the fiber $\pi^{-1} (\pi (\msp))$ which is a finite set of points since $\pi$ is finite.
This shows that $l A$, where $l = \max \{\deg g_k\}$, isolates $\msp$. 
\end{proof}

\begin{Prop} \label{prop:isol}
\begin{enumerate}
\item Let $X$ be a member of family No.~$i$, where 
\[
i \in \{20,24,31,37,45,47,51,59,64,71,75,76,78,84,85\},
\]
and $\msp$ a nonsingular point of $X$.
Then $l A$ isolates $\msp$, where $lA$ is the one described in \emph{Table \ref{table:isol}}.
\item Let $X$ be a member of family No.~$14$ and $\msp$ a nonsingular point of $X$.
Then $2A$ isolates $\msp$.
\end{enumerate}
\end{Prop}

\begin{proof}
First, we prove (1).
For families No.~$20,24,31,37$, we do not need $\msp$-isolating classes when $\msp \notin H_x$, hence the corresponding columns are blank in Table \ref{table:isol}.
We indicate in the 4th column of the table the projection which is a finite morphism.
Here we denote by $\pi_u$ (resp.\ $\pi_{su}$, resp.\ $\pi_{tu}$) the projection with the coordinates other than $u$ (resp.\ $s,u$, resp.\ $t,u$).
If $\pi_u$ (resp.\ $\pi_{su}$, resp.\ $\pi_{tu}$) is assigned, then $u^2 \in F_2$ (resp.\ $s^2 \in F_1$, $u^2 \in F_2$, resp.\ $t^2 \in F_1$, $u^2 \in F_2$) and this immediately implies that the projection is everywhere defined and it does not contract any curve.
The $\msp$-isolating classes given in the table coincides with the ones obtained by Lemma \ref{lem:findisol}.

Next, we prove (2).
Let $X = X_{6,6} \subset \mbP (1,2,2,2,3,3)$ be a member of family No.~14.
Then the projection $\pi \colon X \to \mbP (1,2,2,2) =: \mbP$ is a finite morphism (of degree $4$) and it is clear that $\mcO_{\mbP} (2)$ isolates $\pi (\msp)$.
This shows that $2A$ isolates $\msp$.
\end{proof}

\begin{table}[h]
\begin{center}
\caption{Isolating classes for nonsingular points}
\label{table:isol}
\begin{tabular}{cccc|cccc}
No. & $\msp \notin H_x$ & $\msp \in H_x \setminus H_y$ & $\pi$ & No. & $\msp \notin H_x$ & $\msp \in H_x \setminus H_y$ & $\pi$  \\
\hline
20 & & $3A$ & $\pi_u$ & 64 & $7A$ & $14A$ & $\pi_{su}$ \\
24 & & $4A$ & $\pi_{su}$ &71 & $6A$ & $20A$ & $\pi_{tu}$ \\
31 & & $6A$ & $\pi_u$ & 75 & $8A$ & $8A$ & $\pi_{su}$ \\
37 & & $6A$ & $\pi_{su}$ & 76 & $7A$ & $28A$ & $\pi_{su}$ \\
45 & $5A$ & $10A$ & $\pi_u$ & 78 & $7A$ & $28A$ & $\pi_{tu}$ \\
47 & $4A$ & $12A$ & $\pi_{tu}$ & 84 & $10A$ & $30A$ & $\pi_{su}$ \\
51 & $6A$ & $6A$ &  $\pi_{su}$ & 85 & $10A$ & $72A$ & $\pi_{tu}$ \\
59 & $5A$ & $20A$ & $\pi_{tu}$ & & &
\end{tabular}
\end{center}
\end{table}

For the determination of $\msp$-isolating classes of family No.~$60$, careful arguments are required.

\begin{Lem} \label{lem:tech60}
Let $X$ be a member of family No.~$60$.
\begin{enumerate}
\item Suppose that $X$ contains the curve $\Gamma = (y = z = t = u = 0)$.
Then, $\mult_{\msp} (H_y) \le 2$ for any nonsingular point $\msp \in X$ contained in $\Gamma$.
\item Suppose that $\msp_y \in X$ and $u s \notin F_1$.
Then $\mult_{\msp} (H_x) \le 2$ for any nonsingular point $\msp \in X$.
\item Suppose that $X$ contains the curve $\Gamma = (x = z = t = u = 0)$.
Then $\mult_{\msp} (H_x) \le 2$ for any nonsingular point $\msp \in X$ contained in $\Gamma$.
\end{enumerate}
\end{Lem}

\begin{proof}
We first prove (1).
Let $\msp \in X$ be a nonsingular point contained in $\Gamma$.
We will show that either $\mult_{\msp} (H_y \cdot H_z) \le 2$ or $\mult_{\msp} (H_y \cdot H_t) = 2$, which will imply $\mult_{\msp} (H_y) \le 2$.
Since $\Gamma \subset X$, there is no monomial consisting only of $x$ and $s$ in $F_1, F_2$.
This in particular implies $u s \in F_1$.
Since $t^2 \in F_1$, we may assume that $t^2 x \notin F_2$ by replacing $F_2$ with $F_2 - \theta x F_1$ for a suitable $\theta \in \mbC$.
We may assume that the coefficients of $u s, t^2$ in $F_1$ and $u^2$ in $F_2$ are $1$.
Then we can write
\[
F_1|_{\Pi_{y,z}} = u s + t^2 + \alpha t x^3, \ 
F_2|_{\Pi_{y,z}} = u^2 + \beta u s x + \gamma t x^4,
\]
for some $\alpha,\beta,\gamma$.
If $\gamma \ne 0$, then $H_y \cap H_z$ is nonsingular at $\msp$.
We assume $\gamma = 0$.
Then, $H_y \cdot H_z = \Gamma + \Delta + \Xi$, where 
\[
\Delta = (y = z = u = t + \alpha x^3 = 0), \ \Xi = (y = z = u + \beta s x = F_1|_{\Pi_{y,z}} = 0).
\]
If $\beta \ne 0$, then $\msp \notin \Xi$ and we have $\mult_{\msp} (H_y \cdot H_z) \le 2$.
We assume that $\beta = 0$.
In this case, we have $H_y \cdot H_z = 2 \Gamma + 2 \Delta$.
If $\alpha \ne 0$, then $\msp \notin \Delta$ and we have $\mult_{\msp} (H_y \cdot H_z) = 2$.
We assume $\alpha = 0$.
In this case we have $H_y \cdot H_z = 4 \Gamma$ and $\mult_{\msp} (H_y \cdot H_z) = 4$.

We will show that $H_y \cap H_t$ is nonsingular at $\msp$ assuming $\alpha = \beta = \gamma = 0$.
We may assume that the coefficients of $z^3$ in $F_1$ and $s^2 z$ in $F_2$ are $1$.
We can write
\[
F_1|_{\Pi_{y,t}} = u s + z^3 + \delta_1 z^2 x^2 + \delta_2 z x^4, \ 
F_2|_{\Pi_{y,t}} = u^2 + s^2 z + \varepsilon_1 z^3 x + \varepsilon_2 z^2 x^3 + \varepsilon_3 z x^5,
\]
where $\delta_i, \varepsilon_i \in \mbC$.
If $\varepsilon_3 = 0$, then $H_y \cap H_t$ is nonsingular at any point of $\Gamma \setminus \Sing X$.
We assume $\varepsilon_3 \ne 0$.
If $\msp \ne \msq := (\zeta\!:\!0\!:\!0\!:\!1\!:\!0\!:\!0) \in \Gamma$, where $\zeta$ is a fifth root of $- 1/\varepsilon_3$, then $H_y \cap H_t$ is nonsingular at $\msp$.
Hence we assume $\msp = \msq$.
By setting $s = 1$, we consider the curve $\Sigma$ defined by
\[
u + z^3 + \delta_1 z^2 x^2 + \delta_2 z x^4 = u^2 + z + \varepsilon z^3 x + \varepsilon_2 z^2 x^3 + \varepsilon_3 z x^5 = 0
\]
in $\mbA^4_{x,z,s,u}$.
We have $\mult_{\msp} (H_y \cdot H_t) = \mult_{\breve{\msp}} (\Sigma)$, where $\breve{\msp} = (\zeta,0,0,0)$.
Eliminating $u$, $\Sigma$ is defined by
\[
(z^3 + \delta_1 z^2 x^2 + \delta_2 z x^4)^2 + z (1+ \varepsilon_3 x^5) + \varepsilon_3 z^3 x + \varepsilon_2 z^2 x^2 = 0
\]  
in $\mbA^3_{x,z,s}$ and $\breve{\msp}$ corresponds to $(\zeta,0,0)$.
Replacing $x \mapsto x - \zeta$, the above equation is transformed into $\varepsilon_3 \zeta^4 z x + (\text{other terms}) = 0$ and $\breve{\msp}$ corresponds to the origin.
This shows that $\mult_{\msp} (H_y \cdot H_t) = 2$ as desired.

Next, we prove (2).
The condition $\msp_y \in X$ implies $y^4 \notin F_1$ and $y^2 t \in F_1$.
Hence, we can write
\[
F_1|_{\Pi_{x,z}} = t^2 + t y^2, \ 
F_2|_{\Pi_{x,z}} = u^2 + \alpha t s y + \beta s y^3,
\]
for some $\alpha, \beta \in \mbC$ and this implies $H_y \cdot H_z = \Delta_1 + \Delta_2$, where 
\[
\Delta_1 = (t = u^2 + \beta sy^3 = 0) \cap \Pi_{x,z}, \ \Delta_2 = (t + y^2 = u^2 + \gamma t s y + \beta s y^3 = 0) \cap \Pi_{x,z}.
\] 
We have $\Delta_1 \cap \Delta_2 = \{\msp_s\}$ and $\mult_{\msp} (\Delta_i) \le 2$ for any nonsingular point $\msp \in X$.
Thus $\mult_{\msp} (H_x \cdot H_z) \le 2$ for any nonsingular point $\msp \in X$.

Finally we prove (3).
Since $X$ contains $\Gamma = (x = z = t = u = 0)$, $F_1$ and $F_2$ do not contain a monomial consisting only of $y$ and $s$.
In particular $\msp_y \in X$ and we have $y^2 t \in F_1$.
If $u s \notin F_1$, then the result follows from (2).
Hence we assume $u s \in F_1$.
We can write
\[
F_1|_{\Pi_{x,z}} = u s + t^2 + t y^2, \ 
F_2|_{\Pi_{x,z}} = u^2 + \alpha t s y,
\]
for some $\alpha \in \mbC$.
If $\alpha \ne 0$, then $H_x \cap H_z$ is nonsingular at any point of $\Gamma \setminus \Sing X$.
If $\alpha = 0$, then $H_y \cdot H_z = 2 \Gamma + 2 \Delta$, where $\Delta = (u = t + y^2 = 0) \cap \Pi_{x,z}$.
Since $\Gamma$ is nonsingular and $\Delta \cap \Gamma = \{\msp_s\}$, we have $\mult_{\msp} (H_y \cdot H_z) = \mult_{\msp} (\Gamma) = 1$ for any point of $\Gamma \setminus \Sing X$.
This complete the proof.
\end{proof}

\begin{Prop} \label{prop:isolNo60}
Let $X$ be a member of family No.~$60$ and $\msp \in X$ a nonsingular point.
\begin{enumerate}
\item If $\msp \notin H_x$, then either $lA$ isolates $\msp$ for some $l \le 7$, or $10 A$ isolates $\msp$ and $1 \le \mult_{\msp} (H_y) \le 2$.
\item If $\msp \in H_x \setminus H_y$, then either $lA$ isolates $\msp$ for some $l \le 7$, or $15A$ isolates $\msp$ and $\mult_{\msp} (H_x) \le 2$.
\item If $\msp \in H_x \cap H_y$, then $10A$ isolates $\msp$.
\end{enumerate}
\end{Prop}

\begin{proof}
First, we consider the case $\msp \notin H_x$.
By replacing $z$ and $t$, we can write $\msp = (1\!:\!\lambda\!:\!0\!:\!\mu\!:\!0\!:\!\nu)$ for some $\lambda,\mu,\nu \in \mbC$.
If $\lambda \ne 0$, then we may assume $\mu = \nu = 0$ after replacing $s$ and $u$.
Then the set $\{z,s,t,u\}$, hence $7A$, isolates $\msp$.
If $\lambda = \mu = 0$, then the set $\{y,z,s,t\}$, hence $6A$, isolates $\msp$.
It remains to consider the case when $\lambda = 0$ and $\mu \ne 0$.
We may assume $\nu = 0$ by replacing $u$.
If the set $\{y,z,t,u\}$ isolates $\msp$, then $7A$ isolates $\msp$.
If the above set does not isolates $\msp$, then $\msp \in \Gamma \subset X$, where $\Gamma = (y = z = t = u = 0)$.
In this case the linear system $|\mcI_{\msp} (3 A)|$ consists of the unique member $H_y$ and we have $\mult_{\msp} (H_y) \le 2$ by Lemma \ref{lem:tech60}.

Second, we consider the case $\msp \in H_x \setminus H_y$.
Replacing $t$, we can write $\msp = (0\!:\!1\!:\!\lambda\!:\!\mu\!:\!0\!:\!\nu)$ for some $\lambda,\mu,\nu \in \mbC$.

Suppose that $\lambda \ne 0$.
In this case, we may assume $\nu = 0$ by replacing $u$.
We set $\Pi = \Pi_{x,t,u} \cong \mbP (3_y,4_z,5_s)$.
We can write
\[
F_1|_{\Pi} = \alpha s z y + z^3 + \beta y^4, \ 
F_2|_{\Pi} = s^2 z + \gamma s y^3 + \delta z^2 y^2,
\]
for some $\alpha,\dots,\delta \in \mbC$.
We have the relation
\[
\alpha \mu \lambda + \lambda^3 + \beta = \mu^2 \lambda + \gamma \mu + \delta \lambda^2 = 0
\]
since $\msp \in X$.
In particular $(\alpha,\beta) \ne (0,0)$.
Suppose that $\alpha \ne 0$ and $\beta \ne 0$.
Since $y$ and $z$ does not vanish along $U = (X \cap \Pi) \setminus \{\msp_s\}$, we can find the solutions of $F_1|_{\Pi} = F_2|_{\Pi} = 0$ except for $\msp_s$ by solving the equations
\[
s = - (1/\alpha) (z^3 + \beta y^4), \ 
(z^3 + \beta y^4)^2 - \alpha^2 \gamma y^3 (z^3 + \beta y^4) + \alpha^2 \delta z^3 y^3 = 0
\]
on $U$, where the latter comes from $\alpha^2 z y^2 F_2|_{\Pi}$.
It follows that $X \cap \Pi$ is finite.
Suppose $\alpha = 0$.
In this case $\beta \ne 0$ and it is easy to see that $X \cap \Pi$ is finite, hence $\{x,t,u\}$ isolates $\msp$.
Suppose $\beta = 0$.
In this case $\alpha \ne 0$.
If $\gamma \ne 0$, then $X \cap \Pi$ is finite.
If $\gamma = 0$, then $X \cap \Pi$ is the union of $\Gamma = (x = z = t = u = 0)$ and a finite set of points.
In both cases, $\{x,y,t\}$ isolates $\msp$ since $\msp \notin \Gamma$.
Thus $7A$ isolates $\msp$ and we are in case (a) if $\lambda \ne 0$.

Suppose that $\lambda = 0$ and $\mu \ne 0$.
We set $\Pi = \Pi_{x,z,t} \cong \mbP (3_y,5_s,7_u)$ and write
\[
F_1|_{\Pi} = \alpha u s + \beta y^4, 
F_2|_{\Pi} = u^2 + \gamma s y^3,
\]
for some $\alpha,\beta,\gamma \in \mbC$.
We have
\[
\alpha \mu \nu + \beta = \nu^2 + \gamma \mu = 0
\]
since $\msp \in X$.
If $\beta \ne 0$, then $H_s \cap (X \cap \Pi) = \emptyset$ and this implies that $X \cap \Pi$ is finite since $H_s$ is ample.
In this case $6A$ isolates $\msp$ and we assume $\beta = 0$ in the rest.
If $\alpha = 0$, then $\mult_{\msp} (H_x) \le 2$ by (2) of Lemma \ref{lem:tech60}, and we are in case (b).
We assume $\alpha \ne 0$.
It follows that $\nu = 0$ and $\gamma = 0$.
In this case we have $X \cap \Pi = \Gamma$ set-theoretically, where $\Gamma = (x = z = t = u = 0)$, and we conclude $\mult_{\msp} (H_x) \le 2$ by (3) of Lemma \ref{lem:tech60}.

Suppose that $\lambda = \mu = 0$.
Note that $\nu \ne 0$ since $\msp$ is a nonsingular point.
It is clear that the set $\{x,z,s,t\}$ isolates $\msp$.
Thus $6A$ isolates $\msp$.

Finally, we consider the case $\msp \in H_x \cap H_y$.
We can write $\msp = (0\!:\!0\!:\!1\!:\!\lambda\!:\!\mu\!:\!\nu)$ for some $\lambda,\mu,\nu \in \mbC$.
Since $t^2, z^3 \in F_1$ and $u^2, s^2 z \in F_2$, we can write
\[
F_1|_{\Pi_{x,y}} = \alpha us + t^2 + z^3, \ 
F_2|_{\Pi_{x,y}} = u^2 + \beta t z^2 + s^2 z,
\]
for some $\alpha, \beta \in \mbC$.

If $\mu = 0$, then $\{x,y,t\}$ isolates $\msp$ and thus $6A$ isolates $\msp$.
Hence we assume $\lambda \ne 0$.
The section $g = \lambda^2 z t - \mu s^2$ vanishes at $\msp$ and we set $\Sigma = X \cap (x = y = g = 0) \ni \msp$.
We see $\Sigma \cap H_z = \emptyset$.
This implies that $\Sigma$ cannot contain a curve since $H_z$ is an ample divisor.
Hence $\{x,y,g\}$ isolates $\msp$ and $10A$ is an $\msp$-isolating class.
This complete the proof.
\end{proof}

\subsection{Quadratic involutions of codimension $2$ Fano $3$-folds}
\label{sec:QIcodim2}

Let $X = X_{d_1,d_2} \subset \mbP (a_0,\dots,a_5)$ be a codimension $2$ Fano $3$-fold and let $\msp \in X$ be a singular point which is a center of quadratic involution (see \cite[Section 5.1]{Okada1}).
The aim of this subsection is to understand the divisor $G$ (see Section \ref{sec:Sinvolandlct}) and the flopping curves explicitly.

\begin{Lem}[{\cite[Lemma 5.2]{Okada1}}] \label{lem:QIcoordgen}
We can choose homogeneous coordinates $x_{i_0}$, $x_{i_1}$, $x_{i_2}$, $x_{i_3}$, $\xi$ and $\zeta$ of $\mbP (a_0,\dots,a_5)$ such that $\msp = \msp_{\xi}$ and the defining polynomials $F_1, F_2$ are written as
\[
\begin{split}
F_1 &= \xi^2 x_{i_0} + \xi a + \zeta^2 + b, \\
F_2 &= \xi x_{i_1} + \zeta c + d,
\end{split}
\]
where $a,b,c,d \in \mbC [x_{i_0},x_{i_1},x_{i_2},x_{i_3}]$ are homogeneous polynomials.
Here we do not assume that $d_1 \le d_2$.
\end{Lem}

We fix the above coordinates and let $a_{i_j}$, $a_{\xi}$ and $a_{\zeta}$ be the weights of the coordinates $x_{i_j}$, $\xi$ and $\zeta$, respectively. 
We have $0 < \bar{a}_{\zeta} := a_{\zeta} - a_{\xi} < a_{\xi}$, $a_{i_2}, a_{i_3} < a_{\xi}$ and, up to permutation, the triplet $(a_{i_2}, a_{i_2}, \bar{a}_{\zeta})$ coincides with $(1,k, a_{\xi} - k)$ for some integer $0 < k < a_{\xi}$ such that $k$ is coprime to $a_{\xi}$.
Let $\varphi \colon Y \to X$ be the Kawamata blowup at $\msp$ with exceptional divisor $E$.
By the argument in \cite{Okada1}, $\varphi$ can be identified with the embedded weighted blowup at $\msp$ with weights 
\[
\mathrm{wt} (x_{i_0},x_{i_1},x_{i_2},x_{i_3},\zeta) = \frac{1}{a_{\xi}} (a_{i_0},a_{i_1},a_{i_2},a_{i_3},\bar{a}_{\zeta}),
\]
and we have the natural isomorphism
\begin{equation} \label{eq:EofQI}
E \cong (x_{i_0} + \zeta^2 = x_{i_1} + \zeta c = 0) \subset \mbP (a_{i_0},a_{i_1},a_{i_2},a_{i_3},a_{\zeta}-a_{\xi}) =: \mbP_{\mathrm{exc}}.
\end{equation}
It is proved in \cite{Okada1} that we have the following diagram
\[
\xymatrix{
& Y \ar[dl]_{\varphi} \ar[rd]^{\psi} \ar@{-->}[rrrr]^{\tau} & & & & Y \ar[dr]^{\varphi} \ar[ld]_{\psi} & \\
X \ar@{-->}[rrrd]^{\pi \hspace{3mm}} & & Z \ar[rd] \ar[rr]^{\iota} & & Z \ar[ld] & & X \ar@{-->}[ll]^{\pi} \\
& & & \mbP_{\mathrm{base}} & & &}
\]
where $\psi$ is a $K_Y$-trivial contraction which is not an isomorphism, $\iota \colon Z \to Z$ is the biregular involution associated with the double cover $Z \to \mbP_{\mathrm{base}} = \mbP (a_{i_0},a_{i_1},a_{i_2},a_{i_3})$ and $\tau \colon Y \ratmap Y$ the induced birational involution.
We define $\rho := \pi \circ \varphi \colon Y \to \mbP_{\mathrm{base}}$.
If $\psi$ is small, that is, a flopping contraction, then $\tau$ is the flop and the induced birational involution $\sigma \colon X \ratmap X$ is a Sarkisov link.
Note that the above diagram exists even when $\psi$ is divisorial.

\begin{Rem}
To be rigorous, it is not proved in \cite{Okada1} that $\psi$ is indeed small for a general member $X$ (and this is still enough for the purpose of that paper).
In this paper we will make use of the structure of the Sarkisov link $\sigma$.
Hence we need to show that $\psi$ is small for a general $X$.
\end{Rem}

\begin{Lem} \label{lem:quadinvG'}
If $c \ne 0$ as a polynomial, then the divisor $G' = (x_{i_0} c^2 + x_{i_1}^2 = 0)_X$ satisfies the following.
\begin{enumerate}
\item $\tilde{G}' = \varphi^*\tilde{G} - ((n + a_{\xi})/a_{\xi}) E \sim_{\mbQ} n' B - E$, where $n' = 2 a_{i_1}$.
\item $\psi (\tilde{G}') = (x_{i_0} c^2 + x_{i_1}^2 = 0) \subset \mbP_{\mathrm{base}}$.
\end{enumerate}
\end{Lem}

\begin{proof}
By multiplying $c^2$ to $F_1$, eliminating $\zeta c = - (\xi x_{i_1} + d)$ in terms of the equation $F_2 = 0$, we have
\begin{equation} \label{eq:quadinvhighsec}
\xi^2 (x_{i_0} c^2 + x_{i_1}^2) + \xi (a c^2 + 2 x_{i_1} d) + b c^2 + d^2 = 0.
\end{equation}
Each term in the above equation other than $\xi^2 (x_{i_0} c^2 + x_{i_1})$ vanish along $E$ to order at least $\deg (x_{i_0} c^2 + x_{i_1}^2) + a_{\xi}$.
Then we have $\tilde{G}' = \varphi^*G' - ((n + a_{\xi})/a_{\xi}) E \sim_{\mbQ} 2 a_{i_1} B - E$ and (1) is proved.
(2) is easy since $\psi (\tilde{G}') = \pi (G')$.
\end{proof}

\begin{Lem}
The birational map $\psi$ is a small contraction if and only if $c \ne 0$ as a polynomial.
\end{Lem}

\begin{proof}
If $c = 0$, then the divisor $(x_{i_1} = 0)_X$ is contracted to the curve $(x_{i_1} = d = 0) \subset \mbP_{\mathrm{base}}$ and thus $\psi$ is divisorial.

Suppose that $c \ne 0$.
We will derive a contradiction by further assuming that $\psi$ is divisorial.
Let $F$ be the $\psi$ exceptional divisor.
We first claim that if $D \subset X$ is a prime divisor satisfying $\tilde{D} \sim_{\mbQ} k B - e E$ for some $k, e > 0$, then $\tilde{D} = F$.
Indeed, if $\tilde{D}$ is not contracted by $\psi$, then we have $\psi^*(\psi_*(\tilde{D})) = \tilde{D} + \alpha F$ for some $\alpha \ge 0$.
Since $\psi \colon Y \to Z$ is the anticanonical model and it is divisorial, we have $\psi^*(\psi_*(\tilde{D})) \sim_{\mbQ} k' B$ for some $k' > 0$.
It follows that $\alpha F \sim_{\mbQ} (k' - k) B + e E$ is $\psi$-positive since $B$ is $\psi$-trivial and $E$ is $\psi$-positive.
This is a contradiction.
Now, by Lemma \ref{lem:quadinvG'}, $\tilde{G}' \sim_{\mbQ} n' B - E$.
Hence $\tilde{G}'$ must contain $F$ as a component and this happenes only if $F \sim_{\mbQ} n'' B - E$ for some $0 < n'' \le n'$ and $\tilde{G'} = F + D$ for some effective divisor $D \sim_{\mbQ}  (n' - n'') B$.
Since $\psi(\tilde{G}')$ is a surface, $D \ne 0$, that is, $n'' < n'$. 
We see that $\psi(D)$ contains $\psi (F)$, which implies that $\psi^*(\psi_*(D)) = D + \alpha F$ for some $\alpha > 0$.
This is a contradiction since $D = \psi^*(\psi_*(D)) - \alpha F \not\sim_{\mbQ} (n'-n'') B$ and the proof is completed.
\end{proof}

In the following, we assume that $c \ne 0$.
We define $M := (c = 0)_X$.
We have $\tilde{M} \sim_{\mbQ} (\deg c) B$, where $\tilde{M}$ is the proper transform of $M$ via $\varphi$.
Note that the the birational involution $\sigma$ restricts to the biregular involution $\sigma|_M \colon M \to M$, which is defined by $\zeta \mapsto - \zeta$.
Thus, $\tau_* \tilde{M} = \tilde{M}$.
We set $\Theta := \tilde{M}|_E$ which is an effective divisor on $E$.
Note that
\[
\Theta = (x_{i_0} + \zeta^2 = x_{i_1} = c = 0) \subset E \subset \mbP_{\mathrm{exc}}
\]
and that $\tau_*\Theta = \Theta$ since $\Theta \subset \tilde{M} \cap E$ and, on $\tilde{M}$, $\tau$ is just the involution $\zeta \mapsto - \zeta$.
Let $G \subset X$ be the prime divisor such that $\tilde{G} = \tau_*E$ and let $n_G$ be the integer such that $G \subset |n_G A|$.

\begin{Lem} \label{lem:n_Gcomp}
We have
\[
n_G = \frac{2 (B^2 \cdot E)}{B^3} = 2 a_{i_1}.
\]
\end{Lem}

\begin{proof}
We prove the first equality.
We consider the $1$-cycle $\Theta = \tilde{M} \cdot E$.
For a subset $\Delta \subset Y$, we define $\Delta^{\circ} = \Delta \setminus \Exc (\psi)$.
The involution $\tau$ induces a biregular involution of $Y^{\circ}$, and it maps $E^{\circ}$ and $\tilde{M}^{\circ}$ isomorphically onto $\tilde{G}^{\circ}$ and $\tilde{M}^{\circ}$, respectively.
Moreover we have $\tau (\Theta^{\circ}) = \Theta^{\circ}$.
Thus the $1$-cycle $\tilde{M} \cdot \tilde{G} = \tau_* \tilde{M} \cdot \tau_* E$ coincides with $\Theta = \tau_* \Theta = \tau_* (\tilde{M} \cdot E)$, and we can write $\tilde{M} \cdot \tilde{G} = \Theta + \Gamma$, where $\Gamma$ is supported on $\Exc (\psi)$.
Since each component of $\Gamma$ is contracted by $\psi$, we have $B \cdot \Gamma = 0$.
Hence
\[
B \cdot \tilde{M} \cdot \tilde{G} = B \cdot \Theta + B \cdot \Gamma = B \cdot \Theta.
\]
Now, since $\tilde{G} \sim_{\mbQ} n B - E$, $\tilde{M} \sim_{\mbQ} (\deg c) B$ and $B \cdot \Theta = B \cdot \tilde{M} \cdot E$, we obtain the desired equality.

We prove the second equality.
It is enough to show $a_{i_1} (A^3) = (B^2 \cdot E) + (a_{i_1}/a_{\xi}^3) E^3$ since $B^3 = A^3 - (1/a_{\xi}^3) E^3$.
By looking at the equations $F_1, F_2$, we have $d_1 = 2 a_{\zeta} = 2 a_{\xi} + a_{i_0}$, $d_2 = a_{\xi} + a_{i_1}$.
Note that $a_{i_0} = 2 \bar{a}_{\zeta}$.
Hence
\[
a_{i_1} A^3 = \frac{2 a_{i_1} a_{\zeta} (a_{\xi} + a_{i_1})}{a_{i_0} a_{i_1} a_{i_2} a_{i_3} a_{\xi} a_{\zeta}} 
= \frac{2 (a_{\xi} + a_{i_1})}{a_{i_0} a_{i_2} a_{i_3} a_{\xi}}
= \frac{a_{\xi} + a_{i_1}}{a_{i_2} a_{i_3} a_{\xi} \bar{a}_{\zeta}}.
\]
Note that $E^3 = a_{\xi}^2/a_{i_2} a_{i_3} \bar{a}_{\zeta}$ and we have 
\[
(B^2 \cdot E) + \frac{a_{i_1}}{a_{\xi}^3} E^3 = \frac{a_{\xi} + a_{i_1}}{a_{\xi}^3} E^3 = \frac{a_{\xi} + a_{i_1}}{a_{i_2} a_{i_3} a_{\xi} \bar{a}_{\zeta}},
\]
which completes the proof.
\end{proof}

\begin{Lem} \label{lem:quadinvimportant}
The following assertions hold.
\begin{enumerate}
\item $G = (x_{i_0} c^2 + x_{i_1}^2 = 0) \cap X$.
\item $\tilde{G}|_E = (c (a c - 2 \zeta d) = 0)|_E = \Theta + \Xi$, where $\Theta = (c = 0)$ and $\Xi = (a c - 2 \zeta d = 0)$ are divisors on $E$.
\item $\Exc (\psi) \cap E \subset (b c^2 + d^2 = 0) \cap \tilde{G} \cap E \subset \mbP_{\mathrm{exc}}$.
\end{enumerate}
\end{Lem}

\begin{proof}
We set $G' = (x_{i_0} c^2 + x_{i_1}^2 = 0) \cap X$ as in Lemma \ref{lem:quadinvG'}. 
Then, by Lemma \ref{lem:n_Gcomp}, we have $n' = n_G$, that is, $\tilde{G}' \sim_{\mbQ} n' B - E \sim_{\mbQ} \tilde{G}$.
Since the divisor $\tilde{G} = \tau_*E$ is not movable, we conclude $\tilde{G}' = \tilde{G}$ and (1) is proved.

We have
\[
\begin{split}
\tilde{G}|_E &= (x_{i_0} + \zeta^2 = x_{i_1} + \zeta c = a c^2 + 2 x_{i_1} d = 0) \\
&= (x_{i_0} + \zeta^2 = x_{i_1} + \zeta c = c (a c - 2 \zeta d) = 0) \\
&= (c (a c - 2 \zeta d) = 0)|_E,
\end{split}
\]
which proves (2).

We see that the image $\rho (E)$ is the surface $(x_{i_0} c^2 + x_{i_1}^2 = 0) \subset \mbP_{\mathrm{base}}$, where the defining equation is obtained by eliminating $\zeta$ in the equations of $E \subset \mbP_{\mathrm{exc}}$.
Moreover, since $\tilde{G}|_E = (a c^2 + 2 x_{i_1} d = 0) |_E$, we have 
\[
\Sigma := \rho (\tilde{G}|_E) = (x_{i_0} c^2 + x_{i_1}^2 = a c^2 + 2 x_{i_1} d = 0) \subset \mbP_{\mathrm{base}}.
\]
Since every flopping curve is contained in $\tilde{G}$ and intersects $E$, we have $\rho (\Exc (\psi)) \subset \rho (\tilde{G}|_E) = \Sigma$.
We have
\[
\pi^{-1} (\Sigma) = (x_{i_0} c^2 + x_{i_1}^2 = a c^2 + 2 x_{i_1} d = b c^2 + d^2 = 0) \cap X
\]
since we have the equation \eqref{eq:quadinvhighsec} on $X$, and
\[
\Sigma' := \pi (\pi^{-1} (\Sigma)) = (x_{i_0} c^2 + x_{i_1}^2 = a c^2 + 2 x_{i_1} d = b c^2 + d^2 = 0) \subset \mbP_{\mathrm{base}}.
\]
Therefore we have (3) since $\Exc (\psi) \cap E \subset (\rho|_E)^{-1} (\Sigma')$.
\end{proof}

\section{Nonsingular points} \label{sec:nonsingpts}

The aim of this section is to prove the following.

\begin{Prop} \label{prop:lctnspt}
Let $X$ be a general member of family No.~$i$ with $i \in I^*_{br}$.
Then $\lct_{\msp} (X) \ge 1$ for any nonsingular point $\msp \in X$.
\end{Prop}

This is a combination of Propositions \ref{prop:nonsingI}, \ref{prop:lctnsptII}, \ref{prop:lctnsptNo14} and \ref{prop:lctnsptNo8} below. 

\subsection{Families other than No.~$8$, $14$, $20$, $24$, $31$ and $37$}

\begin{Lem} \label{lem:excl4551}
Let $X$ be a member of family No.~$i$, where $i \in \{45,51\}$, and $\msp \in L_{xy}$ a nonsingular point of $X$. 
Suppose that $L_{xy}$ is either reducible or non-reduced.
Then $\mult_{\msp} (D) \le 1$ for any effective $\mbQ$-divisor $D \sim_{\mbQ} A$.
\end{Lem}

\begin{proof}
Suppose that the conclusion does not hold.
Then there is an irreducible $\mbQ$-divisor $D \sim_{\mbQ} A$ such that $\mult_{\msp} (D) > 1$.
 
Suppose first that $i = 45$.
By Proposition \ref{prop:Lxysp}, $H_x \cdot H_y = \Gamma_1 + \Gamma_2$, where $\Gamma_1, \Gamma_2$ are irreducible nonsingular curves with $A \cdot \Gamma_1 = A \cdot \Gamma_2 = 1/10$.
Since they do not intersect at a nonsingular point, we may assume $\msp \in \Gamma_1$ and $\msp \notin \Gamma_2$.
We write $D \cdot H_x = \gamma_1 \Gamma_1 + \gamma_2 \Gamma_2 + \Delta$, where $\gamma_1, \gamma_2 \ge 0$ and $\Delta$ is an effective $1$-cycle such that $\Gamma_1, \Gamma_2 \not\subset \Supp (\Delta)$.
Since $A$ is ample, we have
\[
\frac{1}{10} = A \cdot D \cdot H_x \ge \frac{1}{10} \gamma_1 + \frac{1}{10} \gamma_2.
\]
We have
\[
\frac{1}{5} - \frac{1}{5} \gamma_1 - \frac{1}{5} \gamma_2 
= H_y \cdot (D \cdot H_x - \gamma_1 \Gamma_1 - \gamma_2 \Gamma_2) = H_y \cdot \Delta \ge \mult_{\msp} \Delta > 1 - \gamma_1.
\]
The above two inequalities imply $\gamma_2 < 0$.
This is a contradiction.

Suppose next that $i = 51$.
By Proposition \ref{prop:Lxysp}, $H_x \cdot H_y = 2 \Gamma$, where $\Gamma$ is an irreducible nonsingular curve with $A \cdot \Gamma = 1/12$.
We write $D \cdot H_x = \gamma \Gamma + \Delta$, where $\gamma \ge 0$ and $\Delta$ is an effective $1$-cycle such that $\Gamma \not\subset \Supp (\Delta)$.
We have
\[
\frac{1}{12} = A \cdot D \cdot H_x \ge \frac{1}{12} \gamma,
\]
and
\[
\frac{1}{6} - \frac{1}{6} \gamma 
= H_y \cdot (D \cdot H_x - \gamma \Gamma) 
= H_y \cdot \Delta > 1 - \gamma.
\]
The above two inequalities derives a contradiction and the proof is completed.
\end{proof}

\begin{Lem} \label{lem:exclnonHx}
Let $X$ be a member of family No.~$i$, where $i \in I^*_{br} \setminus \{8,14\}$, and $\msp \in H_x$ a nonsingular point of $X$.
Then $\mult_{\msp} (D) \le 1$ for any effective $\mbQ$-divisor $D \sim_{\mbQ} A$.
\end{Lem}

\begin{proof}
We first consider the case when 
\[
i \in \{31,37,45,47,51,64,71,75,76,78,84,85\}.
\]
These families satisfy $1 = a_0 < a_1 < a_2$ and $a_1 A^3 \le 1$.

Suppose that $\msp \in L_{xy}$.
We see that $S_1 := H_x \sim_{\mbQ} A$ and $S_2 := H_y \sim_{\mbQ} a_1 A$ passes through $\msp$ and $\mult_{\msp} (H_x) = 1$ since $H_x$ is quasi-smooth.
By Lemma \ref{lem:excl4551} and Proposition \ref{prop:irredLxy}, we may assume that $L_{xy} = H_x \cap H_y$ is irreducible and reduced, and $\mult_{\msp} (L_{xy}) = 1$.
Since the inequality $a_1 A^3 \le 1$ holds, we have $\mult_{\msp} (D) \le 1$ by Lemma \ref{lem:exclL}.
Suppose $\msp \in H_x \setminus L_{xy} = H_x \setminus H_y$.
Then $S := H_x$ is a unique member of $|\mcI_{\msp} (A)|$ and $\mult_{\msp} (H_x) = 1$.
By Proposition \ref{prop:isol}, there is a $\msp$-isolating class $lA$ such that $l A^3 \le 1$.
Thus, by Lemma \ref{lem:exclG}, $\mult_{\msp} (D) \le 1$.

Next, we consider the case when $i \in \{20,24,59\}$.
In this case, $a_0 = 1 < a_1 = a_2 < a_3$.
Suppose that $\msp \in H_x$.
It is easy to see that $\Pi_{x,y,z} \cap X$ is a finite set (of singular points).
This implies that either $\msp \notin H_y$ or $\msp \notin H_z$.
Replacing $y$ and $z$, we may assume $\msp \in H_x \setminus H_y$.
We set $S = H_x$, which satisfies $\mult_{\msp} (S) = 1$.
By Proposition \ref{prop:isol}, $l A$ isolates $\msp$, where $l = 3, 4, 20$ if $i = 20, 24, 59$, respectively, and we have $1 \cdot l \cdot A^3 = 1$.
By Lemma \ref{lem:exclG}, $\mult_{\msp} (D) \le 1$.

Finally we consider the case $i = 60$.
Suppose that $\msp \in H_x \cap H_y$.
By Proposition \ref{prop:isolNo60}, $10A$ isolates $\msp$.
We can apply Lemma \ref{lem:exclG} for $S_1 = H_x, S_2 = H_y$ since $3 \cdot 10 \cdot A^3 = 1$, and we conclude $\mult_{\msp} (D) \le 1$. 
Suppose that $\msp \in H_x \setminus H_y$.
We first show that $\mult_{\msp} (H_x) \le 2$.
If $\mult_{\msp} (H_x) > 2$, then $lA$ isolates $\msp$ for some $l \le 7$ by Proposition \ref{prop:isolNo60}.
Take $R \in |\mcI_{\msp} (6A)| \ne \emptyset$.
Then $H_x \cdot R$ is an effective $1$-cycle and there is a divisor $T \in |\mcI_{\msp}^k (k l A)|$ which does not contain any component of $D\cdot R$ for $k \gg 0$.
We have
\[
\frac{84}{30} k \ge \frac{12 l}{30} k = H_x \cdot R \cdot T \ge 3 k.
\]
This is a contradiction and we have $\mult_{\msp} (H_x) \le 2$.
Then we can apply Lemma \ref{lem:exclG} for $S = H_x$ since $lA$ isolates $\msp$ for some $l \le 15$ and $2 \cdot 15 \cdot A^3 = 1$, and conclude $\mult_{\msp} (D) \le 1$.
\end{proof}

\begin{Lem} \label{lem:No47sp}
Let $X$ be a member of No.~$47$ and $\msp \notin H_x$ a nonsingular point of $X$.
Then, for the unique divisor $R \in |\mcI_{\msp} (3 A)|$, we have $\mult_{\msp} (R) \le 3$.
\end{Lem}

\begin{proof}
Assume that $\mult_{\msp} (R) > 3$.
By Lemma \ref{lem:highmultdiv}, there is a divisor $M \in |\mcI_{\msp}^2 (mA)|$ for some $4 \le m \le 6$.
Since $4 A$ isolates $\msp$ by Proposition \ref{prop:isol}, we can take $T \in |\mcI_{\msp}^k (4 k A)|$ which does not contain any component of $R \cdot M$.
Then we have
\[
6 k \ge 12 m k A^3 = R \cdot M \cdot T > 3 \cdot 2 \cdot  k.
\]
This is a contradiction and we have $\mult_{\msp} (R) \le 3$.
\end{proof}

\begin{Lem} \label{lem:gennsptmult2}
Let $X$ be a general member of family No.~$i \in \{45,51,64,75\}$.
Then, for any nonsingular point $\msp \notin H_x$ of $X$, the unique member of $|\mcI_{\msp} (2 A)|$ has multiplicity at most $2$ at $\msp$.
\end{Lem}

\begin{proof}
We will give a proof for family No.~$45$.
Proofs for the other families are the same.
Let $X = X_{10,12} \subset \mbP (1,2,4,5,5,6) =: \mbP$ be a member of family No.~$45$.
Let $\msp \notin H_x$ be a point of $X$.
We denote by $R_{\msp}$ the unique member of $|\mcI_{\msp} (2 A)|$.
Then we can assume $\msp = \msp_x$ by replacing coordinates.
In this case $R_{\msp} = H_y$ and we can write
\[
\begin{split}
F_1 &= \alpha_1 y x^8 + \beta_1 z x^6 + \gamma_1 s x^5 + \delta_1 t x^5 + \varepsilon_1 u x^4 + \cdots, \\
F_1 &= \alpha_2 y x^{10} + \beta_2 z x^8 + \gamma_2 s x^7 + \delta_2 t x^7 + \varepsilon_2 u x^6 + \cdots,
\end{split}
\]
for some $\alpha_j,\dots,\varepsilon_j \in \mbC$.
We see that $\mult_{\msp} H_y > 1$ if and only if the rank of the matrix
\[
\begin{pmatrix}
\alpha_1 & \beta_1 & \cdots & \varepsilon_1 \\
\alpha_2 & \beta_2 & \cdots & \varepsilon_2 
\end{pmatrix}
\]
is less than $2$, which imposes $3$ conditions for $\alpha_j,\dots,\varepsilon_j$.
Let $\mcF$ be the space parametrizing the (quasi-smooth) members of family No.~$45$ and we define
\[
\mcW_k = \{\, (X, \msp) \mid \msp \in X \in \mcF \text{ and } \mult_{\msp} R_{\msp} \ge k \,\} \subset \mcF \times \mbP.
\]
For a point $\msp \in \mbP$ on which $x \ne 0$, the fibers of $\mcW_1 \to \mbP$ over $\msp$ is of codimension at least $2 + 3 = 5$, where the $2$ comes from the condition $\msp \in X$ and the $3$ comes from the condition $\mult_{\msp} (R_{\msp}) > 1$ as explained above.
It follows that $\dim \mcW_1 \le \dim (\mcF \times \mbP) - 5 = \dim \mcF$.
It is easy to see that $\mcW_2$ is a proper closed subset of $\mcW_1$, that is, for a general point $\msp$ and a general $X \ni \msp$ such that $(X, \msp) \in \mcW_1$, $\mult_{\msp} R_{\msp} = 2$.
This shows $\dim \mcW_2 < \dim \mcF$.
Therefore a general $X$ satisfies (5).
\end{proof}

\begin{Prop} \label{prop:nonsingI}
Let $X$ be a member of family No.~$i$, where 
\[
i \in I^*_{br} \setminus \{8,14,24,31,37\} = \{45,47,51,59,60,64,71,75,76,78,84,85\},
\]
and $\msp \in X$ a nonsingular point.
Then $\mult_{\msp} (D) \le 1$ for any effective $\mbQ$-divisor $D \sim_{\mbQ} A$.
\end{Prop}

\begin{proof}
Let $D \sim_{\mbQ} A$ be an effective $\mbQ$-divisor.
By Lemma \ref{lem:exclnonHx}, it is enough to prove $\mult_{\msp} (D) \le 1$ for $\msp \notin H_x$.

We first consider the case when $i \in \{45,47,51,64,71,75,76,78,84,85\}$.
These families satisfy $1 = a_0 < a_1 < a_2$ and $a_1 A^3 \le 1$.
If $i \in \{45,47,51,64,75\}$, then we define $S$ to be the unique member of $|\mcI_{\msp} (a_1A)|$.
By Lemmas \ref{lem:No47sp} and \ref{lem:gennsptmult2}, we have $\mult_{\msp} (S) \le a_1$.
We observe that the $\msp$-isolating class given in Proposition \ref{prop:isol} satisfies $a_1 l A^3 \le 1$.
If $i \in \{71,76,78,84,85\}$, then we define $S$ to be a general member of the linear system $|\mcI_{\msp} (a_2A)|$ which is of positive dimensional.
The $\msp$-isolating class given in Proposition \ref{prop:isol} satisfies the inequality $a_2 l A^3 \le 1$.
Thus, by Lemma \ref{lem:exclG}, $\mult_{\msp} (D) \le 1$. 

Next we consider the case when $i = 59$.
In this case $X = X_{12,14} \subset \mbP (1,4,4,5,6,7)$.
Suppose that $\msp \notin H_x$.
Let $S$ be a general member of the pencil $|\mcI_{\msp} (4A)|$.
By Proposition \ref{prop:isol}, $5A$ isolates $\msp$ and we have $4 \cdot 5 \cdot  A^3 = 1$.
By Lemma \ref{lem:exclG}, $\mult_{\msp} (D) \le 1$.
This finishes the proof.

Finally we consider the case when $i = 60$.
If $\msp \notin H_y$, then $lA$ isolates $\msp$ for some $l \le 7$ by Proposition \ref{prop:isolNo60}.
In this case we can apply Lemma \ref{lem:exclG} for $S \in |\mcI_{\msp} (4 A)| \ne \emptyset$ since $4 l A^3 \le 1$, and we have $\mult_{\msp} (D) \le 1$.
In the following, we assume $\msp \in H_y$.
We claim that $\mult_{\msp} (H_y) \le 3$.
Assume to the contrary that $\mult_{\msp} (H_y) \ge 4$.
We take $R \in |\mcI_{\msp} (4 A)| \ne \emptyset$.
Then $H_y \cdot R$ is an effective $1$-cycle with $\mult_{\msp} (H_y \cdot R) \ge 4$.
By Proposition \ref{prop:isolNo60}, $lA$ isolates $\msp$ for some $l \le 7$ and we can take $T \in |\mcI_{\msp}^k (k l A)|$ for $k \gg 0$ which does not contain any component of $H_y \cdot R$.
We have
\[
\frac{84}{30} k \ge \frac{12 l}{30} kH_y \cdot R \cdot T \ge 4 k.
\]
This is a contradiction and we have $\mult_{\msp} (H_y) \le 3$.
Now we can apply Lemma \ref{lem:exclG} for  $S = H_y$ since $10A$ isolates $\msp$ and $3 \cdot 10 \cdot A^3 = 1$, and conclude $\mult_{\msp} (D) \le 1$.
\end{proof}

\subsection{Families No.~$20$, $24$, $31$, $37$}

Let $X$ be a codimension $2$ Fano $3$-fold embedded in $\mbP (1,a_1,\dots,a_5)$.
Suppose that $\msp_x \in X$.
Then, by a choice of coordinates, the defining polynomials are written as
\[
F_1 = x^{e_1} w_1 + G_1, \ F_1 = x^{e_2} w_2 + G_2,
\]
where $e_1,e_2 > 0$, $\{w_1,w_2\} \subset \{y,z,s,t,u\}$ and $G_1, G_2 \in (y,z,s,t,u)^2 \subset \mbC [x,y,\dots,u]$.
In this case, we call the pair $(w_1,w_2)$ the {\it tangent pair} of $X$ at $\msp$ (with respect to the fixed choice of coordinates).
Also, we call $H_{w_1}$, $H_{w_2}$ the {\it tangent divisors} of $F_1$, $F_2$ at $\msp$, respectively. 
Note that tangent pair of $X$ depends on the choice of coordinates.
However, if there is another choice of coordinates for which the tangent pair of $X$ at $\msp_x$ is $(w'_1,w'_2)$, then $\deg w_1 = \deg w'_1$ and $\deg w'_1 = \deg w'_2$.
For pairs of coordinates $(w_1,w_2)$ and $(w'_1,w'_2)$, we say that they are equivalent if $\deg w_j = \deg w'_j$ for $j = 1,2$.
We denote by $[w_1,w_2]$ the equivalence class of $(w_1,w_2)$.
Then, we say that $X$ is of {\it type} $[w_1,w_2]$ at $\msp$ if there is a choice of coordinates for which $(w_1,w_2)$ is the tangent pair of $X$ at $\msp$.
For a point $\msp$ of $X$ which is not contained in $H_x$, we say that $X$ is of {\it type} $[w_1,w_2]$, where $\{w_1,w_2\} \subset \{y,z,s,t,u\}$, if after replacing coordinates, the point $\msp$ is transformed into $\msp_x$ and $X$ is of type $[w_1,w_2]$ at $\msp_x$.

\begin{Lem} \label{lem:lctnsptlowdeg1}
Let $X$ be a member of family No.~$i \in \{20,24\}$ and $\msp \notin H_x$ a nonsingular point of $X$.
Suppose that $X$ is either of type $[w_1,w_2]$ for some coordinates $w_1, w_2$ with $\{w_1,w_2\} \subset \{s,t,u\}$, of type $[w,z]$ for some coordinate $w \in \{s,t\}$, or of type $[z,u]$.
Then $\mult_{\msp} (M) \le 2$ for any $M \in |\mcI_{\msp} (2A)|$.
\end{Lem}

\begin{proof}
We replace coordinates so that $\msp = \msp_x$.
In this case we have $|\mcI_{\msp} (2A)| = \{y,z\}$.

Suppose that $X$ is of type $[w_1,w_2]$ at $\msp$, where $\{w_1,w_2\} \subset \{s,t,u\}$.
Then it is easy to see that any divisor $M \in |\mcI_{\msp} (2A)|$ is nonsingular at $\msp$, that is, $\mult_{\msp} (M) = 1$, and the assertion follows.

Suppose that $X$ if of type $[w,z]$ at $\msp$, where $w \in \{s,t\}$.
Then, since $u$ can be chosen as a part of local coordinates of $X$ at $\msp$ and $u^2 \in F_2$, we have $\mult_{\msp} (H_z) = 2$.
Now the assertion follows since $\mult_{\msp} (M) = 1$ for $M \in |\mcI_{\msp} (2A)| \setminus \{H_z\}$.

Finally, suppose that $X$ if of type $[z,u]$ at $\msp$.
We can choose $y,s,t$ as local coordinates of $X$ at $\msp$.
If $i = 20$, then at least one of $s^2, st, t^2$ is contained in $F_1$, and if $i = 24$, then $s^2 \in F_1$.
This shows that $\mult_{\msp} (H_z) = 2$ and the assertion follows since $\mult_{\msp} (M) = 1$ for $M \in |\mcI_{\msp} (2A)| \setminus \{H_z\}$.
\end{proof}

\begin{Lem} \label{lem:lctnsptlowdeg2}
Let $X$ be a general member of family No.~$i \in \{20,24\}$.
Then, $\mult_{\msp} (M) \le 2$ for any nonsingular point $\msp \notin H_x$ of $X$ and any divisor $M \in |\mcI_{\msp} (2 A)|$.
\end{Lem}

\begin{proof}
We set $U = \mbP (1,2,2,3,a_4,a_5) \setminus (x=0)$.
We will count the number of conditions imposed in order for $X$ to contain a point $\msp \in U$ such that there exists a divisor $M \in |\mcI_{\msp} (2A)|$ satisfying $\mult_{\msp} (M) > 3$.
Replacing coordinates, we may assume $\msp = \msp_x$.

In order for $X$ to contain a point $\msp \in U$ for which $X$ is of type $[z,t]$, $[z,s]$ or $[z,y]$, at least $6$ conditions are imposed.
For example, in the case of $i = 20$ and $\msp$ is of type $[z,t]$, we have $6$ conditions $u x^2, t x^3, s x^3 x^6 \notin F_1$ and $u x^4, x^8 \notin F_2$.
Thus $X$ does not contain a point $\msp \notin L_{xy}$ such that it is one of the above types.

By Lemma \ref{lem:lctnsptlowdeg1} and the above argument, it remains to consider the case when $X$ contains $\msp \in U$ as a type $[u,z]$ point.
We see that $4$ conditions $x^{d_1} \notin F_1$, $t x^{d_2-a_4}, s x^{d_2-a_3}, x^{d_2} \notin F_2$, where $d_j = \deg F_j$, are imposed in order for $X$ to be of type $[u,z]$ at $\msp$.
Now suppose that $X$ contains a point $\msp \in U$ as a type $[u,z]$ point.
It is clear that $\mult_{\msp} (M) = 1$ for $M \in |\mcI_{\msp} (2A)| \setminus \{H_z\}$.
Since $y,s,t$ can be chosen as local coordinates of $X$ at $\msp$, $\mult_{\msp} (H_z) > 2$ if and only if the quadratic part of $F_2 (1,y,0,s,t,0)$ is zero, and the latter imposes additional $6$ conditions.
Therefore, by the dimension counting argument, the assertion follows for a general $X$.
\end{proof}

\begin{Lem} \label{lem:lctnsptlowdeg3}
Let $X$ be a general member of family No.~$i \in \{31,37\}$ and $\msp \notin H_x$ a nonsingular point of $X$.
Then the following hold.
\begin{enumerate}
\item $\mult_{\msp} (M) \le 2$ for the unique divisor $M \in |\mcI_{\msp} (2A)|$.
\item $\mult_{\msp} (M) \le 3$ for any divisor $M \in |\mcI_{\msp} (3A)|$.
\end{enumerate}
\end{Lem}

\begin{proof}
Let $\msp \in U = \mbP (1,2,3,a_3,a_4,a_5) \setminus (x = 0)$ be a nonsingular point of $U$.
We will count the number of conditions in order for $X$ to admit a divisor $M$ with multiplicity greater than $2$ or $3$ at $\msp$.
Without loss of generality, we may assume $\msp = \msp_x$.

We prove (1).
It is clear that if $X$ contains $\msp$ as a type $[w_1,w_2]$ point, where $\{w_1,w_2\} \subset \{s,t,u\}$, then $\mult_{\msp} (M) = 1$.
Suppose that $X$ contains $\msp$ as a $[w,y]$ type point, where $w \in \{z,s,t\}$.
We claim that $\mult_{\msp} (M) = 2$ in this case.
Note that $M = H_y$.
Since $u$ can be chosen as a part of local coordinates of $X$ at $\msp$ and $u^2 \in F_2$, we have $\mult_{\msp} (M) = 2$ and the claim is proved.
It is easy to see that at least $6$ conditions are imposed in order for $X$ to contain $\msp$ as a type $[y, w]$ point for some $w \in \{s,t,u\}$.
Thus it remains to consider the case when $X$ contains $\msp$ as a type $(u,y)$ point, which imposes $5$ conditions.
We have $M = H_y$, and we have $\mult_{\msp} (M) > 2$ if and only if the quadratic part of $F_2 (1,z,s,t,0)$ is zero, which imposes $6$ additional conditions.
Therefore (1) is proved.

We prove (2).
We have $|\mcI_{\msp} (3A)| = \langle z,yx \rangle$.
By (1), we have $\mult_{\msp} (M') \le 2$, where $M' = (y x = 0)_X$.
We set $\mcM = |\mcI_{\msp} (3A)| \setminus \{M'\}$.
It is clear that if $X$ contains $\msp$ as a type $[w_1,w_2]$ point, where $\{w_1,w_2\} \subset \{y, s,t,u\}$, then any divisor $M \in \mcM$ is nonsingular at $\msp$, that is, $\mult_{\msp} (M) = 1$.
We see that at least $6$ conditions are imposed in order for $X$ to contain $\msp$ as a type $[z,s]$ or $[z,t]$ point.
It is also easy to see that if $X$ contains $\msp$ as a type $[w,z]$ point, where $w \in \{s,t\}$, then $\mult_{\msp} (M) = 2$ for any $M \in \mcM$ since $u^2 \in F_2$.
Thus it remains to consider the case when $X$ contains $\msp$ as a type $[u,z]$, $[z,u]$, $[z,t]$ or $[z,s]$ point.
It is now straightforward to count the number of conditions in each case (in fact, we can even prove that $\mult_{\msp} (M) \le 2$ for any $M \in |\mcI_{\msp} (3A)|$ although we do not need this fact) and the proof is completed.
\end{proof}

\begin{Prop} \label{prop:lctnsptII}
Let $X$ be a member of family No.~$i \in \{20,24,31,37\}$ and $\msp \in X$ a nonsingular point.
Then $\mult_{\msp} (D) \le 1$ for any effective $\mbQ$-divisor $D \sim_{\mbQ} A$.
\end{Prop}

\begin{proof}
We assume that the conclusion does not hold.
Then there is an irreducible $\mbQ$-divisor $D \sim_{\mbQ} A$ such that $(X, D)$ is not log canonical at $\msp$.
By Lemma \ref{lem:exclnonHx}, we have $\msp \notin H_x$, and we may assume $\msp = \msp_x$ by a coordinate change.

Note that, by Lemmas \ref{lem:lctnsptlowdeg2} and \ref{lem:lctnsptlowdeg3}, $\Supp (D) \ne M$ for any $M \in |\mcI_{\msp} (n A)|$, where $n = a_1, a_2$.
We claim that there is no divisor $M \in |\mcI_{\msp} (nA)|$, where $n = a_1,a_2$, such that $\mult_{\msp} (D \cdot M) > 2$.
Indeed, suppose that such a divisor $M$ exists.
Then, since $a_4 A$ isolates $\msp$, we can take a divisor $T \in |\mcI_{\msp}^k (k a_4 A)|$ which does not contain any component of $D \cdot M$.
It follows that
\[
2 k \ge k n a_4 (A^3) = D \cdot M \cdot T >  k\mult_{\msp} (D \cdot M) > 2 k,
\]
This is a contradiction and the claim is proved.

In the following, we say that a set $\Sigma \subset X$ is $\msp$-{\it finite} if $\Sigma \ni \msp$ and it does not contain a curve through $\msp$.
Let $\varphi \colon Y \to X$ be the blowup at $\msp$ with exceptional divisor $E \cong \mbP^2$.
Note that $a_2 a_4 (A^3) = 2$ and that there is a line $L \subset E$ with the property that $\mult_{\msp} (D \cdot M) > 2$ for an effective divisor $M$ such that $\tilde{M} \supset L$.

Suppose that $X$ is of type $[y,w]$, $[w,y]$, $[z,w]$ or $[w,z]$ for some coordinate $w \in \{y,z,s,t,u\}$.
Then there is a divisor $M \in |\mcI_{\msp} (n A)|$ such that $n \in \{a_1,a_2\}$ and $\mult_{\msp} (M) \ge 2$ (For example, if $X$ is of type $(y,w)$, then $M = H_y$ satisfies the above property).
This implies $\mult_{\msp} (D \cdot M) > 2$, which is impossible.

Suppose that $X$ is of type $[u,t]$ or $[t,u]$ at $\msp$.
We have $E \cong \mbP^2_{y,z,s}$ and $L = (\lambda s + \mu z + \nu y = 0)$.
If $\lambda = 0$, then there exists a divisor $M \in |\mcI_{\msp} (nA)|$ with $n \in \{a_1,a_2\}$ such that $\mult_{\msp} (D \cdot M) > 2$ (if $\mu \ne 0$ and $\mu = 0$, then take $M = (\mu z + \nu y x^{a_2-a_1} = 0)_X$ and $M = H_y$, respectively).
Thus $\lambda \ne 0$.
Then we set $M = (\lambda s + \mu z x^{a_3-a_2} + \nu y x^{a_3-a_1} = 0)_X$.
We have $\Supp (D) \ne M$ and $\mult_{\msp} (D \cdot M) > 2$ since $\mult_{\msp} (M) = 1$ and $\tilde{M} \supset L$.
We see that $M \cap \Bs |\mcI_{\msp} (a_2 A)|$ is a $\msp$-finite set.
It follows that, for a general member $T \in |\mcI_{\msp} (a_2 A)|$,
\[
2 \ge a_2 a_3 (A^3) = D \cdot M \cdot T > 2.
\]
This is impossible.

Suppose that $X$ is of type $[u,s]$ or $[s,u]$.
We have $E \cong \mbP^2_{y,z,t}$ and $L = (\lambda t + \mu z + \nu y = 0)$.
By the same reason as above, we have $\lambda \ne 0$.
We set $M = (\lambda t + \mu z x^2 + \nu y x^3 = 0)_X$.
Then $\Supp (D) \ne M$ and $\mult_{\msp} (D \cdot M) > 2$ since $\mult_{\msp} (M) = 1$ and $\tilde{M} \supset L$.
We see that $M \cap \Bs |\mcI_{\msp} (a_2 A)|$ is a $\msp$-finite set.
It follows that, for a general member $T \in |\mcI_{\msp} (a_2 A)|$,
\[
2 = a_2 a_4 (A^3) = D \cdot M \cdot T > 2.
\]
This is impossible.

Suppose that $X$ is of type $[t,s]$ or $[s,t]$ at $\msp$.
Let $M$ be a general member of the pencil $\langle t, s x^{a_4-a_3} \rangle$, so that $\Supp (D) \ne M$.
Note that $\mult_{\msp} (M) = 2$ and the set $M \cap \Bs |\mcI_{\msp} (a_2A)|$ is $\msp$-finite.
Thus, for a general $T \in |\mcI_{\msp} (a_2 A)|$, we have
\[
2 = a_2 a_4 (A^3) = D \cdot M \cdot T > 2.
\]
This is impossible.
The above arguments exhaust all the cases, and we have derived a contradiction.
\end{proof}

\subsection{Family No.~14}

\begin{Prop} \label{prop:lctnsptNo14}
Let $\msp \in X$ be a nonsingular point of $X$.
Then, $\mult_{\msp} (D) \le 1$ for any effective $\mbQ$-divisor $D \sim_{\mbQ} A$.
\end{Prop}

\begin{proof}
If $\msp \in H_x$, then we can apply (2) of Lemma \ref{lem:exclG} for $S = H_x$ since $2 A$ is an isolating class by Proposition \ref{prop:isol} and we have $1 \cdot 2 \cdot A^3 = 1$. 

Assume that the conclusion does not hold.
Then there is a point $\msp \notin H_x$ and an irreducible $\mbQ$-divisor $D$ on $X$ such that $\mult_{\msp} (D) > 1$. 

We claim that there exists a divisor $S \in |\mcI_{\msp} (2 A)|$ such that $D \cdot S$ is an effective $1$-cycle satisfying $\mult_{\msp} (D \cdot S) > 2$.
We may assume that $\msp = \msp_x$ and we can write
\[
F_1 = \alpha u x^3 + \beta t x^3 + \cdots, \ 
F_2 = \gamma u x^3 + \delta t x^3 + \cdots,
\]
for some $\alpha,\beta,\gamma,\delta \in \mbC$.
Suppose that the matrix
\[
\begin{pmatrix}
\alpha & \beta \\
\gamma & \delta
\end{pmatrix}
\]
is of rank less than $2$.
Then, possibly interchanging $F_1, F_2$ and replacing coordinates, we may assume $u x^3, t x^3 \notin F_1$.
Since $X$ is nonsingular at $\msp$, we may assume $F_1 = s x^3 + G_1$, where $G_1 = G_1 (x,y,z,s,t,u) \in (y,z,s,t,u)^2$, after replacing $y,z,s$.
We set $S = H_s$.
Then we have $\mult_{\msp} (S) = 2$, hence $D \ne S$ and $\mult_{\msp} (D \cdot S) > 2$ as desired. 
Suppose that the above matrix is of rank $2$.
Then we may assume $F_1 = t x^3 + \cdots$ and $F_2 = u x^3 + \cdots$.
Let $\varphi \colon Y \to X$ be the blowup at $\msp$ with exceptional divisor $E \cong \mbP^2$.
By \cite[Corollary 3.5]{Corti}, there is a line $L \subset E$ with the property that $\mult_{\msp} (D \cdot S) > 2$ for any effective divisor $S$ such that $\tilde{S} \supset L$ and $\Supp (S) \not\supset \Supp (D)$.
We see that $F$ is naturally isomorphic to the plane defined by $t = u = 0$ in $\mbP^4$ with coordinates $y,z,s,t,u$.
It follows that $L$ is defined by $t = u = \lambda y + \mu z + \nu s = 0$ for some $\lambda,\mu,\nu \in \mbC$ with $(\lambda,\mu,\nu) \ne (0,0,0)$.
Let $S$ be the divisor on $X$ which is cut out by $\lambda y + \mu z + \nu s = 0$.
We see that $S$ is nonsingular at $\msp$, hence $\Supp (D) \ne S$ and $D \cdot S$ is an effective $1$-cycle.
Moreover the proper transform $\tilde{S}$ contains $L$ and the existence of $S$ is proved.

By Proposition \ref{prop:isol}, $2 A$ isolates $\msp$ and thus, for $k \gg 0$, we can take a divisor $T \in |\mcI_{\msp}^k (2 k A)|$ which does not contain any component of $D \cdot S$.
Therefore
\[
2 k = D \cdot S \cdot T \ge \mult_{\msp} (D \cdot S) \mult_{\msp} (T) > 2 k.
\]
This is a contradiction and the proof is completed. 
\end{proof}

\subsection{Family No.~$8$}

Let $X = X_{4,6} \subset \mbP (1,1,2,2,2,3)$ be a member of family No.~$8$ and $\msp \notin L_{xy}$ a nonsingular point of $X$.
We denote by $S_{\msp}$ the unique member of $|\mcI_{\msp} (A)|$.

\begin{Lem} \label{lem:No8nonsing1}
Suppose that $X$ is a general member of family No.~$8$.
Then we have $\lct_{\msp} (X,S_{\msp}) = 1$ for any nonsingular point $\msp \notin L_{xy}$ of $X$.
\end{Lem}

\begin{proof}
Let $\msp$ be any point of $\mbP (1,1,2,2,2,3)$ such that $\msp \notin (x = y = 0)$.
It is easy to see that $5$ conditions are imposed in order for $X$ to contain $\msp$ and $\mult_{\msp} (S_{\msp}) > 1$.
Thus we cannot conclude that $\mult_{\msp} (S_{\msp}) = 1$ for any point $\msp \notin L_{xy}$.
However, by the above dimension count, we can assume that a general $X$ does not contain a point $\msp \notin L_{xy}$ satisfying $\mult_{\msp} (S_{\msp}) > 1$ and some further additional conditions.
For example, we can assume that $(\prt F_1/\prt u) (\msp) \ne 0$ for any point $\msp \in X$ such that $\msp \notin L_{xy}$ and $\mult_{\msp} (S_{\msp}) > 1$.

Now let $\msp \notin L_{xy}$ be a point of $X$ such that $\mult_{\msp} (S_{\msp}) > 1$.
Replacing coordinates, we may assume $\msp = \msp_x$, and in this case $S_{\msp} = H_y$.
We can write
\[
F_1 = u x - f_4, \ 
F_2 = y x^5 + u^2 + u g_3 + h_6,
\]
where $f_4, g_3, h_6 \in \mbC [x,y,z,s,t]$, $f_4, h_6 \in (y,z,s,t)^2$ and $g_3 \in (y,z,s,t)$. 
We write $h_6 = x^2 q (z,s,t) + c (z,s,t) + y h_5$, where $q$ and $c$ are quadratic and cubic forms, respectively, and $h_5 \in \mbC [x,y,z,s,t]$.
If $X$ is general, then we can assume that $q$ is of rank $3$, that is, $(q = 0) \subset \mbP^2_{z,s,t}$ is nonsingular.
We work on $U_x = X \setminus H_x$.
By setting $x = 1$ and eliminating $u$ in terms of $F_1|_{U_x} = 0$, $U_x$ is isomorphic to the hypersurface in $\mbA^4_{y,z,s,t}$ defined by
\[
y + f_4 (1,y,z,s,t)^2 + f_4 (1,y,z,s,t) g_3 (1,y,z,s,t) + h_6 (1,y,z,s,t) = 0.
\]
By filtering off terms divisible by $y$, we can write
\[
(-1 + \cdots) y = f_4 (1,0,z,s,t)^2 + q (z,s,t) + c (z,s,t),
\]
where the omitted terms $\cdots$ are non-zero monomials in variables $y,z,s,t$.
The projectivised tangent cone of the right-hand side is isomorphic to $(q = 0) \subset \mbP^2_{z,s,t}$ and it is nonsingular.
Thus, by \cite[8.10 Lemma]{Kol}, $\lct_{\msp} (X,H_y) = 1$. 
\end{proof}

\begin{Lem} \label{lem:No8spnsptgen}
Suppose that $X$ is a general member of family No.~$8$.
If $X$ contains a WCI curve $\Gamma$ of type $(1,2,2,3)$, then the following assertions hold.
\begin{enumerate}
\item For any nonsingular point $\msp \in \Gamma$ of $X$, $X$ is of type $[u,t]$ at $\msp$, and in particular, the tangent divisor $T_{\msp}$ for $F_2$ at $\msp$ satisfies $T_{\msp} \in |2A|$.
\item For any nonsingular point $\msp \in \Gamma$ of $X$, $T_{\msp} |_{S_{\msp}} = \Gamma + \Delta$, where $\Delta$ is an irreducible and reduced curve such that $A \cdot \Delta = 3/2$ and $\mult_{\msp} (\Delta) = 1$.
\item The surface $S_{\msp}$ is quasi-smooth.
\end{enumerate}
\end{Lem}

\begin{proof}
We see that the WCI curves of type $(1,2,2,3)$ on $\mbP$ form a $7$-dimensional family.
Let $\Gamma$ be a WCI curve of type $(1,2,2,3)$.
We will count the number of condition imposed in order for $X$ to contain $\Gamma$.
Without loss of generality, we may assume $\Gamma = (y = s = t = u = 0)$.
Then $X$ contains $\Gamma$ if and only if $z^2, z x^2, x^4 \notin F_1$ and $z^3, z^2 x^2, z x^4, x^6 \notin F_2$.
Thus $7$ conditions are imposed.
Suppose that $X \supset \Gamma = (y=s=t=u=0)$.
Any nonsingular point $\msp$ of $X$ contained in $\Gamma$ satisfies $\msp \notin H_x$, and we see that $X$ is not of type $[u,t]$ if and only if either $u x \notin F_1$ or $t x^4, s x^4 \notin F_2$.
This imposes at least $1$ additional condition.
Thus, (1) holds for a general $X$.

Let $\msp \in \Gamma = (y = z = s = t = 0) \subset X$ be a nonsingular point.
Since a general member of $|A|$ is quasi-smooth, there are at most finitely many non-quasi-smooth members of $|A|$.
Thus we can assume that the surface $S_{\msp}$, which is the unique member of $|A|$ containing $\Gamma$, is quasi-smooth.
This shows (3).
We will verify (2).
We may assume that $\msp = \msp_x$ and that the tangent pair of $X$ at $\msp$ is $(u,t)$ after replacing $z$.
Note that $S_{\msp} = H_y$ and $T_{\msp} = H_t$.
Then $T_{\msp}|_{S_{\msp}}$ is a curve defined in $\mbP (1_x,2_z,2_s,3_u)$ by the equations,
\[
u x^3 + \ell_1 (z,s) s = u^2 + q (z,s) s + \ell_2 (z,s) s x^2 = 0,
\]
where $\ell_1, \ell_2$ are linear forms and $q$ is a quadratic form in $z,s$.
We work on the open set $U_x \subset \mbP (1,2,2,3)$ on which $x \ne 1$ by setting $x = 1$.
Then, by eliminating $u = - \ell_1 s$, $T_{\msp} |_{S_{\msp}}$ is defined by
\[
s (\ell_1^2 s + q + \ell_2) = 0.
\]
It is easy to see that the curve defined by $\ell_1^2 s + q + \ell_2 = 0$ is irreducible and reduced, and is nonsingular at $\msp$ for a general combination of $\ell_1,\ell_2,q$.
In other words, at least one condition is imposed in order for $X$ to contain $\Gamma$ such that there is $\msp \in \Gamma$ for which the assertion (2) fails.
Therefore (2) holds for a general $X$. 
\end{proof}

\begin{Lem} \label{lem:No8nonsing2}
Let $\msp \notin L_{xy}$ be a nonsingular point of $X$ and set $S = S_{\msp}$.
Suppose that there is an irreducible $\mbQ$-divisor $D \sim_{\mbQ} A$ such that $(X,D)$ is not log canonical at $\msp$.
Then one of the following hold.
\begin{enumerate}
\item $D \ne S$, $\mult_{\msp} (D \cdot S) \ge 2$ and $S \cap \Bs |\mcI_{\msp} (2A)| = \{\msp\}$.
\item There is a divisor $M \in |\mcI_{\msp} (2A)|$ such that $D \ne M$, $\mult_{\msp} (D \cdot M) > 2$ and either $D \cap M \cap S$ does not contain a curve or the $1$-dimensional component of $D \cap M \cap S$ is an irreducible and reduce curve $\Gamma$ of degree $1$ satisfying $\mult_{\msp} (\Gamma) = 1$.
\item The surface $S$ is nonsingular at $\msp$ and there is a divisor $M \in |\mcI_{\msp} (2 A)|$ such that $M|_S = \Gamma_1 + \Gamma_2$, where $\Gamma_1, \Gamma_2$ are irreducible and reduced curves of degree $1$ satisfying $\mult_{\msp} (\Gamma_i) = i$, $(\Gamma_i^2)_S = 0$ and $(\Gamma_1 \cdot \Gamma_2)_S = 2$.
\item The surface $S$ is quasi-smooth and there is an effective $\mbQ$-divisor $C \sim_{\mbQ} A|_S$ on $S$ such that $(S, C)$ is not log canonical at $\msp$ and $C = \gamma \Gamma + \delta \Delta + \Xi$, where $\gamma, \delta$ are non-negative rational number with $\gamma \delta = 0$ and $\Gamma, \Delta, \Xi$ are effective divisors on $S$ such that $\Gamma$ is irreducible and nonsingular with $A \cdot \Gamma = 1/2$, $\Delta$ is irreducible with $A \cdot \Delta = 3/2$, $\Gamma, \Delta \not\subset \Supp (\Xi)$, $(\Gamma^2)_S = -3/2$, $(\Gamma \cdot \Delta)_S = 5/2$ and $(\Delta^2)_S = 1/2$.
\end{enumerate}
\end{Lem}

\begin{proof}
We may assume $\msp = \msp_x$ and in this case $S = H_y$.
We can write
\[
F_1 = \zeta u x + a_4 + b_2 x^2 + y G_1, \ 
F_2 = u^2 + \eta u x^3 + c_6 + d_4 x^2 + e_2 x^4 + y G_2,
\]
where $\zeta, \eta \in \mbC$, $a_4,b_2,c_6,d_4,e_2 \in \mbC  [z,s,t]$ are homogeneous polynomials of indicated degrees, and $G_1, G_2 \in \mbC [x,y,z,s,t,u]$.
We denote by $\varphi \colon Y \to X$ the blowup of $X$ at $\msp$ and by $E \cong \mbP^2$ its exceptional divisor.
By the assumption, $(X,D)$ is not log canonical at $\msp$.
Hence there is a line $L \subset E$ such that $\mult_{\msp} (D \cdot M) > 2$ for an effective $\mbQ$-divisor such that $\tilde{M} \supset L$ and $D \not\subset \Supp (M)$.

Suppose that $\zeta \ne 0$.
Then, we may assume $\zeta = 1$ and $\eta = 0$ by replacing $F_2$ with $F_2 - \theta x^2 F_1$ for a suitable $\theta \in \mbC$.
Replacing $u$ with $u - b_2 x$, we may assume $b_2 = 0$.
If $e_2 = 0$, then $x^5 \in G_2$ since $X$ is nonsingular at $\msp = \msp_x$.
This implies $\mult_{\msp} (S) \ge 2$.
By Lemma \ref{lem:No8nonsing1}, $\lct_{\msp} (X,S) = 1$, which implies $D \ne S$.
It is clear that $S \cap \Bs |\mcI_{\msp} (2 A)| = \{\msp\}$ since $|\mcI_{\msp} (2 A)| = \langle y^2,z,s,t \rangle$, and we are in case (1).
We assume that $e_2 \ne 0$.
Then we may assume $e_2 = t$ after replacing $z,s,t$.
Now $F_1$ and $F_2$ are transformed into the following forms:
\[
F_1 = u x + a_4 + y G_1, \ 
F_2 = u^2 + c_6 + d_4 x^2 + t x^2 + y G_2.
\] 
In this case, $u, t$ are the tangent coordinates of $X$ at $\msp$, and we have the natural isomorphism
\[
E \cong (t = u = 0) \subset \mbP^4_{y,z,s,t,u},
\]
and 
\[
L = (t = u = \lambda y + \mu z + \nu s = 0).
\]
If $(\mu,\nu) = (0,0)$, then $\tilde{S} = \tilde{H}_y$ contains $L$, hence $\mult_{\msp} (D \cdot S) > 2$ and we are in case (1).
In the following we assume $(\mu,\nu) \ne (0,0)$.
Then replacing $z$ or $s$, we may assume $L = (t = u = s = 0)$.
Let $\mcM \subset |\mcI_{\msp} (2A)|$ be the pencil generated by $s$ and $t$.
Let $\alpha,\gamma,\delta \in \mbC$ be the coefficients of $z^2, z^3,z^2$ in $a_4, c_6, d_4$, respectively.
Then
\[
\Sigma := S \cap \Bs \mcM = (y = s = t = u x + \alpha z^2 = u^2 + \gamma z^3 + \delta z^2 x = 0).
\]
If $\alpha \ne 0$, then $\Sigma$ does not contain a curve since $H_x$ is ample and $\Sigma \cap H_x = \emptyset$.
If $\alpha = 0$ and $(\gamma,\delta) \ne (0,0)$, then $\Sigma$ is a finite set.
In both cases, we are in case (2).

We keep the above setting, and consider the case when $\alpha = \gamma = \delta = 0$.
In this case, $\Sigma$ contains the WCI curve $\Gamma = (y = s = t = u = 0)$ of type $(1,2,2,3)$.
We set $T := H_t \sim_{\mbQ} 2 A$.
Then, by Lemma \ref{lem:No8spnsptgen}, we have $T|_S = \Gamma + \Delta$, where $\Delta$ is an irreducible curve on $S$ such that $A \cdot \Delta = 3/2$ and $\mult_{\msp} (\Delta) = 1$.
We write $D|_S = \gamma \Gamma + \delta \Delta + \Xi$, where $\gamma,\delta \ge 0$ and $\Xi$ is an effective divisor on $S$ which does not contain neither $\Gamma$ nor $\Delta$ in its support.
Clearly $(S,D|_S)$ is not log canonical at $\msp$ while $(S, \frac{1}{2} T|_S)$ is log canonical at $\msp$.
For a rational number $\lambda$ with $0 < \lambda \le 1$, we set 
\[
C_{\lambda} := \frac{1}{\lambda} D|_T - \frac{1-\lambda}{\lambda} \left( \frac{1}{2} T|_S \right) 
= \frac{2 \gamma - 1 + \lambda}{2 \lambda} \Gamma + \frac{2 \delta - 1 + \lambda}{2 \lambda} \Delta + \frac{1}{\lambda} \Xi.
\]
Note that $C_{\lambda} \sim_{\mbQ} A|_S$ and we have
\[
D|_T = \lambda C_{\lambda} + (1-\lambda) \left(\frac{1}{2} T|_S \right).
\]
This implies that $(S, C_{\lambda})$ is not log canonical at $\msp$ if $C_{\lambda}$ is effective.
Now it remains to show that there exists $\lambda$ such that $0 < \lambda \le 1$, $C_{\lambda}$ is effective and it does not contain one of $\Gamma$ and $\Delta$ in its support. 
Since $D|_T \sim_{\mbQ} \frac{1}{2} T|_S$, we have
\[
\left (\frac{1}{2} - \gamma \right) \Gamma + \left(\frac{1}{2} - \delta \right) \Delta \sim_{\mbQ} \Xi.
\]
Moreover the case $\gamma = \delta = 1/2$ cannot happen since $D|_T \ne \frac{1}{2} T|_S$.
This implies that either $\gamma \le \delta$ and $\gamma < 1/2$ or $\gamma > \delta$ and $\gamma > \delta$ and $\delta < 1/2$.
If we are in the former case (resp.\ latter case), then we take $\lambda = 1 - 2 \gamma$ (resp.\ $\lambda = 1 - 2 \delta$).
It is straightforward to see that $C_{\lambda}$ is effective and it does not contain one of $\Gamma$ and $\Delta$ in its component.
Finally we compute various intersection numbers of $\Gamma, \Delta$ on $S$.
Since $\Gamma$ is a nonsingular rational curve passing through exactly one singular point of type $\frac{1}{2} (1,1)$ and $K_S \sim_{\mbQ} 0$, we have
\[
(\Gamma^2)_S = -2 + (K_S \cdot \Gamma) + \frac{1}{2} = - \frac{3}{2}.
\]
The computations $(\Gamma \cdot \Delta) = 5/2$ and $(\Delta^2)_S = 1/2$ follows by taking intersection numbers of $2 A|_S \sim_{\mbQ} T|_S = \Gamma + \Delta$ and $\Gamma$, $\Delta$.
Thus, we are in case (4).

In the following, we consider the case when $\zeta = 0$.
Note that, by Lemma \ref{lem:No8spnsptgen}, $X$ does not contain a WCI curve of type $(1,2,2,3)$ passing through $\msp$.

Suppose that $\eta \ne 0$.
We may assume that $\eta = 1$, $b_2 = t$, $e_2 = 0$ and that $x^3 \notin G_1, x^5 \notin G_2$ after replacing coordinates.
Thus we have
\[
F_1 = a_4 + t x^2 + y G_1, \ 
F_2 = u^2 + u x^3 + c_6 + d_4 x^2 + y G_2.
\]
In this case we have $E \cong (t = u = 0) \subset \mbP^4$ and $L = (t = u = \lambda y + \mu z + \nu s = 0)$ for some $\lambda,\mu,\nu \in \mbC$.
If $(\mu,\nu) = (0,0)$, then we are in case (1) and we assume $(\lambda,\mu) \ne (0,0)$.
Replacing, $z$ or $s$, we may assume $L = (t = u = s = 0)$ and let $\mcM \subset |\mcI_{\msp} (2 A)|$ be the pencil generated by $s$ and $t$.
We have
\[
\Sigma := S \cap \Bs \mcM = (y = s = t = \alpha z^2 = u^2 + u x^3 + \gamma z^3 + \delta z^2 x^2 = 0).
\]
If $\alpha \ne 0$, then $\Sigma$ is a finite set and we are in case (2) by taking a general member $M \in \mcM$. 
We assume $\alpha = 0$.
If $(\gamma,\delta) = (0,0)$, then $\Sigma$ contains the WCI curve $(y = s = t = u = 0)$ of type $(1,2,2,3)$, which is impossible.
Hence $(\gamma,\delta) \ne (0,0)$.
In this case $\Sigma$ is an irreducible and reduced curve such that $\mult_{\msp} (\Sigma) = 1$.
Thus we are in case (2) by taking a general member $M \in \mcM$ and setting $\Gamma := \Sigma$.

Finally, suppose that $\eta = 0$.
In this case, replacing $z,s,t$, we may assume that $b_2 = t$, $e_2 = s$ and $x^3 \notin G_1, x^5 \notin G_2$.
Thus $F_1$ and $F_2$ are transformed into the following forms:
\[
F_1 = a_4 + t x^2 + y G_1, \ 
F_2 = u^2 + c_6 + d_4 x^2 + s x^2 + y G_2.
\]
Let $\mcM \subset |\mcI_{\msp} (2A)|$ be the linear system generated by $s$ and $t$.
Note that $\mult_{\msp} (M) \ge 2$ for any $M \in \mcM$.
If $z^2 \in a_4$, then we see that 
\[
S \cap \Bs \mcM = (y = s = t = 0) \cap X
\] 
is a finite set, and we are in case (2).
We assume that $z^2 \notin a_4$ and let $M \in \mcM$ be a general member which is cut out on $X$ by the equation $t - \lambda s = 0$ for some $\lambda \in \mbC$.
Set $\bar{F}_i = F_i (x,0,z,s,\lambda s,u)$.
Then 
\[
\bar{F}_1 = \alpha s (s - \beta z), \ 
\bar{F}_2 = u^2 + \bar{c}_6 + \bar{d}_4 x^2 + s x^2,
\]
for some $\alpha,\beta \in \mbC$ and $\bar{c}_6, \bar{d}_4 \in \mbC [z,s]$.
We can choose generasl $\lambda$ so that $\alpha \ne 0$ and $\beta \ne 0$.
We set $\Gamma_1 = (y = t - \lambda s = \bar{F}_2 = 0)$ and $\Gamma_2 = (y = t - \lambda s = s = \bar{F}_2 = 0)$.
We see that $\Gamma_i$ is irreducible and reduced because otherwise $X$ contains a WCI curve of type $(1,2,2,3)$.
It is easy to see that they are of degree $1$ and $\mult_{\msp} (\Gamma_i) = i$ for $i = 1,2$.
It is also easy to compute that $\Gamma_1 \cap \Gamma_2 = \{\msp\}$ set-theoretically and the local intersection number $(\Gamma_1 \cdot \Gamma_2)_{\msp} = 2$.
It follows that $(\Gamma_1 \cdot \Gamma_2)_S = 2$.
By taking intersection number of $\Gamma_1$ and $T|_S = \Gamma_1 + \Gamma_2$, we have $(\Gamma_1^2) = 0$.
Similarly we have $(\Gamma_2^2)_S = 0$.
Thus we are in case (3) and the proof is completed.
\end{proof}

\begin{Prop} \label{prop:lctnsptNo8}
Let $X$ be a member of family No.~$8$ and $\msp \in X$ a nonsingular point.
Then $\lct_{\msp} (X) \ge 1$.
\end{Prop}

\begin{proof}
Suppose that $\msp \in L_{xy} = H_x \cap H_y$.
By Lemma \ref{prop:irredLxy}, $L_{xy}$ is an irreducible and reduced curve such that $\mult_{\msp} (L_{xy}) = 1$.
Thus the assertion follows from Lemma \ref{lem:exclL} by setting $S_1 = H_x, S_2 = H_y$ since $1 \cdot 1 \cdot A^3 = 1$.

Now we assume that the conclusion does not hold.
Then there is an irreducible $\mbQ$-divisor $D \sim_{\mbQ} A$ such that  $\mult_{\msp} (D) > 1$.

Suppose that we are in case (1) of Lemma \ref{lem:No8nonsing2}.
Then we can take a divisor $M \in |\mcI_{\msp} (2A)|$ which does not contain any component of $D \cdot S$, and we have
\[
2 = D \cdot S \cdot M \ge \mult_{\msp} (D \cdot S) > 2.
\]
This is impossible.

Suppose that we are in case (2) of Lemma \ref{lem:No8nonsing2}.
Let $S, M, \Gamma$ be as in the lemma. 
We can write $D \cdot M = \gamma \Gamma + \Delta$, where $\gamma \ge 0$ and $\Delta$ is an effective $1$-cycle such that $\Gamma \not\subset \Supp (\Delta)$.
If $D \cap M \cap S$ does not contain a curve, then we understand that $\Gamma = 0$ and in the following argument $\gamma = 0$.
Since $\mult_{\msp} (\Delta) = \mult_{\msp} (D \cdot M) - \gamma \mult_{\msp} (\Gamma)$, we have
\[
2 - \gamma = S \cdot (D \cdot M - \gamma \Gamma) = S \cdot \Delta \ge \mult_{\msp} (\Delta) > 2 - \gamma.
\]
This is impossible.

Suppose that we are in case (3) of Lemma \ref{lem:No8nonsing2}.
We can write $D|_S = \gamma_1 \Gamma_1 + \gamma_2 \Gamma_2 + \Delta$, where $\gamma_1, \gamma_2 \ge 0$ and $\Delta$ is an effective $1$-cycle such that $\Gamma_j \not\subset \Supp (\Delta)$ for $j = 1,2$.
We have $1 = A \cdot D \cdot S \ge \gamma_1$.
On the other hand, we have
\[
2 - 2 \gamma_1 - 2 \gamma_2 = M \cdot (D|_S - \gamma_1 \Gamma_1 - \gamma_2 \Gamma_2) = M \cdot \Delta > 1 - \gamma_1 - 2 \gamma_2,
\]
which implies $\gamma_1 < 1$.
This is impossible.

Finally, suppose that we are in case (4) of Lemma \ref{lem:No8nonsing2}.
Let $C = \gamma \Gamma + \delta \Delta + \Xi$ be the effective $\mbQ$-divisor on $S$.
We have $\gamma \delta = 0$.
Suppose that $\gamma = 0$.
Then, since $\Gamma \not\subset \Supp (C)$, we have
\[
\frac{1}{2} = (C \cdot \Gamma)_S > 1.
\]
This is impossible.
Suppose that $\delta = 0$.
We have
\[
\frac{3}{2} - \frac{5}{2} \gamma = (\Delta \cdot (C - \gamma \Gamma))_S= (\Delta \cdot \Xi)_S > 1-\gamma,
\]
and this implies $\gamma < 1/3$.
It follows that the pair $(S, \Gamma + \Xi)$ is not log canonical at $\msp$.
By inversion of adjunction on log canonicity, the pair $(\Gamma, \Xi|_{\Gamma})$ is not log canonical at $\msp$ and we have
\[
\frac{1}{2} + \frac{3}{2} \gamma = (\Gamma \cdot (C-\gamma \Gamma))_S = (\Gamma \cdot \Xi)_S \ge \mult_{\msp} (\Xi|_{\Gamma}) > 1.
\]
This implies $\gamma > 1/3$ and this is impossible.

Therefore we have derived a contradiction and the assertion is proved.
\end{proof}

\section{Singular points} \label{sec:singpts}

The aim of this section is to prove the following.

\begin{Prop} \label{prop:lctsingpt}
Let $X$ be a member of family No.~$i$, where $i \in I^*_{br}$, and let $\msp \in X$ be a singular point.
Then $\lct_{\msp} (X) \ge 1$.
\end{Prop}

This is a combination of Propositions \ref{prop:singptI}, \ref{prop:singptII} and \ref{prop:lctsingIII} below.

\subsection{Singular points with $B^3 \le 0$}

\begin{Prop} \label{prop:singptI}
Let $X$ be a member of family No.~$i \in I^*_{br}$ and $\msp \in X$ a singular point such that $B^3 \le 0$.
Then $\lct_{\msp} (X) \ge 1$.
\end{Prop}

\begin{proof}
If $a_0 = 1$, that is, $X$ is a member of a family other than No.~$60$, then we have $B^2 \notin \Int \bNE (Y)$ by \cite{Okada1} and a general member $S \in |A|$ lifts to an anticanonical divisor $\tilde{S} \in |B|$.
Thus the assertion follows from Lemma \ref{lem:singptNE}.

Let $X$ be a member of family No.~$60$.
If $\msp$ is either of type $\frac{1}{3} (1,1,2)$ or $\frac{1}{5} (1,2,3)$, then we have $\tilde{H}_x \sim_{\mbQ} 2 B$, and the assertion follows from Lemma \ref{lem:singptNE}.
It remains to consider the case when $\msp$ is of type $\frac{1}{2} (1,1,1)$.
By Lemma \ref{lem:60singptnef} below, $N = 5 \varphi^*A - \frac{1}{2} E$ is a nef divisor on $Y$.
If $\msp \in H_x$, then we set $S_1 = H_x$ and $S_2 = H_y$.
We have $\tilde{S}_1 \sim_{\mbQ} 2 \varphi^*A - E$ and $\tilde{S}_2 \sim_{\mbQ} 3 \varphi^*A - \frac{1}{2} E$.
We can apply Lemma \ref{lem:singptnef} for $N, S_1, S_2$ and we have $\lct_{\msp} (X) \ge 1$.
If $\msp \notin H_x$, then we set $S_1 = H_y$ and let $S_2$ be the unique member of $|4A|$ passing through $\msp$.
We have $\tilde{S}_1 \sim_{\mbQ} 3 \varphi^*A - \frac{1}{2} E$ and $\tilde{S}_2 \sim_{\mbQ} 4 \varphi^*A - E$.
By Lemma \ref{lem:singptnef}, $\lct_{\msp} (X) \ge 1$.
This finishes the proof.
\end{proof}

The following result is used in the above proof.

\begin{Lem} \label{lem:60singptnef}
Let $X = X_{12,14} \subset \mbP (2,3,4,5,6,7)$ be a member of family No.~$60$ and $\msp \in X$ a $\frac{1}{2} (1,1,1)$ point.
Then $5B+2 E = 5 \varphi^*A - \frac{1}{2} E$ is nef.
\end{Lem}

\begin{proof}
Assume that $\msp \notin H_x$.
After replacing $z$ and $t$, we may assume $\msp = \msp_x$.
We see $X \cap \Pi_{y,z,s,t} = \{\msp\}$ since $u^2 \in F_2$ and thus the set $\{y,z,s,t\}$ isolates $\msp$.
We see $\ord_E (y,z,s,t) = \frac{1}{2} (1,2,1,2)$.
Hence $5B+2 E = 5 \varphi^*A - \frac{1}{2} E$ is nef.

Assume that $\msp \in H_x$.
We see $z^3, t^2 \in F_1$ by quasismoothness of $X$ and we assume that those coefficients are both $1$.
This implies $\msp = (0\!:\!0\!:\!-1\!:\!0\!:\!1\!:\!0)$.
The condition $\msp \in X$ implies $t z^2 \notin F_2$.
Now we see $X \cap \Pi_{x,y,s} = \{\msp\}$ and hence $\{x,y,s\}$ isolates $\msp$.
We have $\ord_E (x,y,s) \ge \frac{1}{2} (2,1,1)$.
Thus $5B+2 E = 5 \varphi^*A - \frac{1}{2} E$ is nef.
\end{proof}

\subsection{Singular points with $B^3 > 0$, Part 1}

The computations of the global log canonical thresholds of singular points with $B^3 > 0$ can be done easily when $A^3$ is relatively small.

\begin{Prop} \label{prop:singptII}
Let $X$ be a member of family No.~$i \in \{45,51,64,75,76\}$ and $\msp \in X$ be a singular point such that $B^3 > 0$.
Then $\lct_{\msp} (X) \ge 1$.
\end{Prop}

\begin{proof}
We set $S_1 = H_x \sim_{\mbQ} A$, $S_2 = H_y \sim_{\mbQ} a_1 A$ and let $L = S_1 \cap S_2 = L_{xy}$.
By Proposition \ref{prop:irredLxy}, $L$ is irreducible and reduced, and we can easily deduce $\mult_{\breve{\msp}} (\breve{L}) \le 2$ by looking at the equation given in Table \ref{table:Lxy}.
We see $2 \le a_1$ and $r a_1 A^3 \le 1$ for each instance, where $r$ is the index of the point $\msp$.
Thus we verified the conditions in Lemma \ref{lem:singptB3neg} and the assertion follows.
\end{proof}

\subsection{Singular points with $B^3 > 0$, Part 2}

In this subsection, we compute the global log canonical threshold of members of family No.~$i \in \{8,20,24,31,37\}$ at singular points with $B^3 > 0$ by making use of quadratic involutions.

In the following let $X$ be a general member of family No.~$i \in \{8,20,24,31,37\}$ and $\msp \in X$ a singular point with $B^3 > 0$.

\begin{Lem}
By a choice of coordinates, the defining polynomials $F_1, F_2$ of $X$ can be written as follows
\[
\begin{split}
F_1 &= t x_{i_1} + u c + d, \\
F_2 &= t^2 x_{i_0} + t a + u^2 + b,
\end{split}
\]
where $\deg F_1 < \deg F_2$, $a,b,c,d \in \mbC [x_{i_0},\dots,x_{i_3}]$ and $\{x_{i_0},\dots,x_{i_3}\} = \{x,y,z,s\}$ are indicated in table \emph{\ref{table:lctQIcoord}}.
\end{Lem}

Note that the above expression of defining polynomials coincides with those given in Lemma \ref{lem:QIcoordgen} after identifying $\xi = t$, $\zeta = u$ and interchanging the roles of $F_1, F_2$, so that the various descriptions given in Section \ref{sec:QIcodim2} holds true.
As generality conditions, we in particular assume that $c \ne 0$ is general so that, by a choice of coordinates, it is the polynomial indicated in table \ref{table:lctQIcoord}.
The number $d_{\Xi}$ in table \ref{table:lctQIcoord} is the positive integer for which $\Xi \in |\mcO_E (d_{\Xi})|$.

\begin{table}[htb]
\begin{center}
\caption{Codimension $2$ Fano $3$-folds}
\label{table:lctQIcoord}
\begin{tabular}{clcccccc}
No. & Type of $\msp$ & $(x_{i_0}, x_{i_1})$ & $n_G$ & $E$ & $c$ &$d_{\Xi}$ \\[0.5mm]
\hline \\[-3.5mm]
8 & $\frac{1}{2} (1,1,1)$ & $(z,s)$ & $4$ & $\mbP (1_x,1_y,1_u)$ & $x$ & $5$ \\[0.6mm]
20 & $\frac{1}{3} (1,1,2)$ & $(y,s)$ & $6$  & $\mbP (1_x,2_y,1_u)$ & $z$ & $7$ \\[0.6mm]
24 & $\frac{1}{4} (1,1,3)$ & $(y,z)$ & $4$ & $\mbP (1_x,3_s,1_u)$ & $x$ & $7$ & \\[0.6mm]
31 & $\frac{1}{4} (1,1,3)$ & $(y,s)$ & $8$ & $\mbP (1_x,3_z,1_u)$ & $z$ & $9$ & \\[0.6mm]
37 & $\frac{1}{5} (1,1,4)$ & $(y,z)$ & $6$ & $\mbP (1_x,4_s,1_u)$ & $y + \lambda^2 x^2$ & $9$ &
\end{tabular}
\end{center}
\end{table}

\begin{Lem} \label{lem:QIexcgen}
\begin{enumerate}
\item $G \setminus \{\msp\}$ does not contain a singular point with $B^3 > 0$.
\item $\Exc (\psi) \cap E$ does not contain the unique singular point of $E$.
\item The divisors $\Theta$ and $\Xi$ does not share a common component.
\item If $i \ne 37$, then the divisor $\Theta$ is an irreducible nonsingular curve, and if $i = 37$, then $\Theta = \Theta_1 + \Theta_2$, where $\Theta_1$ and $\Theta_2$ are distinct irreducible nonsingular curves such that $\Theta_1 \cap \Theta_2$ consists only of the unique singular point of $E$.  
\item If $i = 8$, then $\Xi$ is an irreducible nonsingular curve, and if $i \ne 8$, then the divisor $\Xi$ on $E$ is quasi-smooth at the unique singular point of $E$ if it passes through the point.
\end{enumerate}
\end{Lem}

\begin{proof}
Let $X$ be a member of family No.~$i$.
We prove (1).
If $i \in \{24,37\}$, then the point $\msp$ is the unique singular point with $B^3 > 0$ and there is nothing to prove.
Suppose that $i = 8$.
It is enough to show that $G \cap \Pi_{x,y,u} = \{\msp\}$.
Since $G \subset X$ is defined by the polynomial $x_{i_0} c^2 + x_{i_1}^2 = z c_1^2 + s^2$, we have
\[
G \cap \Pi_{x,y,u} = (x = y = s = u = \delta z^2 = t^2 z + \beta z^3 = 0).
\] 
where $\delta, \beta \in \mbC$ are the coefficients of $z^2, z^3$ in $d_4$, $b_6$, respectively.
We have $\delta \ne 0$ since $F_1 (0,0,z,s,t,0)$ is of rank $3$.
Hence $G \cap \Pi_{x,y,u} = \{\msp\}$.
Suppose that $i \in \{20,31\}$.
It is enough to show that $G \cap \Pi_{x,y,z,u} = \{\msp\}$.
This can be easily verified since $G$ is defined by $y c^2 + s^2 = 0$, and (1) is proved.

We prove (2).
By Lemma \ref{lem:quadinvimportant}, we have
\[
\Exc (\psi) \cap E \subset (b c^2 + d^2 = c (ac-2ud) = x_{i_0} + u^2 = x_{i_1} + u c = 0) \subset \mbP_{\mathrm{exc}},
\] 
and thus it is enough to check that the set on the right-hand side avoid the unique singular point of $E$, which holds true if and only if $z^6 \in b_8 c_2^2 + d_6^2$, $s^2 \in d_6$, $z^2 \in a_6$ and $s^2 \in d_8$ in the case when $i = 20,24,31$ and $37$, respectively.

The assertion (3) can be verified easily by observing that 
\[
\begin{split}
\Theta \cap \Xi &= (c = a c - 2 u d = x_{i_0} + u^2 = x_{i_1} + u c = 0) \\
&= (c = u = x_{i_0} = x_{i_1} = 0) \cup (c = d = x_{i_0} + u^2 = x_{i_1} = 0)
\end{split}
\]
is a finite subset of $\mbP_{\mathrm{exc}}$.

To observe (4), (5), it is useful to identify $E$ with the weighted projective space indicated in table \ref{table:lctQIcoord} by eliminating $x_{i_0} = - u^2$ and $x_{i_1} = - u c$.
In the following, for a polynomial in $x_{i_0},\dots,x_{i_3}$, we set $\bar{f} = f(-u^2, - u \bar{c},x_{i_2},x_{i_3})$.
Here, note that $\bar{c} = c$ if $i \ne 37$ and $\bar{c} = -u^2 + \lambda^2 x^2$ if $i = 37$.
Under this identification, $\Theta = (\bar{c} = 0)$ and $\Xi = (\bar{a} \bar{c} - 2 u \bar{d} = 0)$ inside $E$.

The assertion (4) is now immediate, and we prove (5). 
If $i = 31$, then $\Xi = (\bar{a}_6 z - 2 u \bar{d}_8 = 0)$ does not pass through the singular point $(0\!:\!1\!:\!0) \in E$ if $z^2 \in a_6$, which holds true for a general $X$.
If $i = 20,24,37$, then $\Xi = (\bar{a} \bar{c} - 2 u \bar{d} = 0)$ is quasi-smooth at the singular point of $E$ if $z^3 \in d_6$, $s^2 \in d_6$, $s^2 \in d_8$, respectively, which hold true for a general $X$.
Finally suppose that $i = 8$.
Then $\Xi = (\bar{a}_4 x - 2 u \bar{d}_4 = 0) \subset \mbP (1_x,1_y,1_u)$.
We see that if $X$ is a general member, then the curve $\Xi$ is a general member of a suitable sub linear system $\mcL \subset |\mcO_{\mbP^2} (5)|$ such that $\Bs \mcL = \{(0\!:\!1\!:\!0)\}$ ($\mcL$ contains $\langle x^5,u^5,y^4 x,y^4 u\rangle$ as a sub linear system).
It follows that $\Xi$ is an irreducible nonsingular curve of degree $5$ in $\mbP^2$. 
\end{proof}

\begin{Prop} \label{prop:lctsingIII}
Let $X$ be a general member of family No.~$i \in \{8,20,24,31,37\}$ and $\msp \in X$ a singular point with $B^3 > 0$.
Then $\lct_{\msp} (X) \ge 1$.
\end{Prop}

\begin{proof}
Let $X$ be a member of family No.~$i$.
Clearly we have $n = n_G \ge 2$. 
By Lemma \ref{lem:QIexcgen}.(1) and by what we have proved, we have $\lct_{G \setminus \{\msp\}} \ge 1$.
Thus, by Proposition \ref{prop:lctselflink}, it is enough to check that $(Y,\frac{1}{n} (\tilde{G} + E))$ is log canonical at any point of $E \cap \Exc (\psi)$.
Let $\msq \in E \cap \Exc (\psi)$, which is a nonsingular point of $Y$ and of $E$ by Lemma \ref{lem:QIexcgen}.(2).
It is then enough to show that $\mult_{\msq} (\frac{1}{n} (\tilde{G} + E)) \le 1$, which is equivalent to $\mult_{\msq} (G) \le n-1$.
From now on, we will prove $\mult_{\msq} (\tilde{G}|_E) \le n-1$, which will imply $\mult_{\msq} (\tilde{G}) \le n-1$.

Suppose that $i = 8$.
In this case $E \cong \mbP^2$, $\Theta$ is a line, and $\Xi \in |\mcO_E (5)|$ is irreducible and nonsingular.
Thus we have 
\[
\mult_{\msq} (\tilde{G}|_E) = \mult_{\msq} (\Theta) + \mult_{\msq} (\Xi) \le 1 + 1 = 2,
\]
which implies $\mult_{\msq} (\tilde{G}|_E) \le n-1$ since $n = 6$.

Suppose that $i \ne 8$.
Let $d_{\Xi}$ be the positive integer such that $\Xi \in |\mcO_E (d_{\Xi})|$.
By Lemma \ref{lem:QIexcgen}.(5) and Lemma \ref{lem:easycompwpp} below, we have 
\[
\mult_{\msq} (\Xi) \le 1 + \lrd \frac{d_{\Xi}-1}{r-1} \rrd =: N.
\]
By case-by-case checking, we have $N \le n-1$ for all the cases and we have $N \le n-2$ for $i \in \{20,31,37\}$.
Thus, if $i \in \{20,31,37\}$, then
\[
\mult_{\msq} (\tilde{G}|_E) = \mult_{\msq} (\Theta) + \mult_{\msq} (\Xi) = 1 + N \le n-1,
\]
and we obtained the desired inequality.

Finally, we consider the case $i = 24$.
If $\msq \notin \Theta$, then
\[
\mult_{\msq} (\tilde{G}|_E) = \mult_{\msq} (\Theta) + \mult_{\msq} (\Xi) = N \le n-1 = 3,
\]
and we are done.
Suppose that $\msq \in \Theta \cap \Xi$.
In this case $\Theta \in |\mcO_E (1)|$ and $\Xi \in |\mcO_E (7)|$, where $E \cong \mbP (1,3,1)$.
The inequality
\[
\frac{7}{3} = \Theta \cdot \Xi \ge \mult_{\msq} (\Xi)
\]
shows $\mult_{\msq} (\Xi) \le 2$, and thus we have the desired inequality $\mult_{\msq} (\tilde{G}|_E) \le 3$.
\end{proof}

\begin{Lem} \label{lem:easycompwpp}
Let $\Delta \in |\mcO_{\mbP (1,1,e)} (m)|$ be a divisor on $\mbP := \mbP (1,1,e)$ and $\msq \in \mbP$ a nonsingular point.
Suppose that $e \ge 2$ and $\Delta$ is quasismooth at $(0\!:\!0\!:\!1)$ if it passes through the point.
Then, 
\[
\mult_{\msq} (\Delta) \le 1 + \lrd \frac{m-1}{e} \rrd.
\]
\end{Lem}

\begin{proof}
There is a unique curve $L \in |\mcO_{\mbP} (1)|$ passing through $\msq$, which is irreducible and nonsingular.
We write $\Delta = k L + \Delta'$, where $k \ge 0$ and $\Delta' \in |\mcO_{\mbP} (m-k)|$ is an effective divisor which does not contain $L$.
If $k \ge 2$, then $\Delta$ cannot be quasi-smooth at $(0\!:\!0\!:\!1)$ since $L$ passes through the point.
Thus we have $k \le 1$ and 
\[
\frac{m-k}{e} = \Delta' \cdot L \ge \mult_{\msq} (\Delta'),
\]
which implies $\mult_{\msp} (\Delta') \le \lrd (m-k)/e \rrd$.
Thus
\[
\mult_{\msq} (\Delta) = k + \mult_{\msq} (\Delta') \le k + \lrd \frac{m-k}{e} \rrd \le 1 + \lrd \frac{m-1}{e} \rrd.
\]
since $k = 0,1$.
\end{proof}

\section{Birationally rigid Pfaffian Fano 3-folds} \label{sec:Pfaff}

There are $3$ families of birationally rigid codimension $3$ Fano $3$-folds.
Such a Fano $3$-fold is a Pfaffian Fano $3$-fold, i.e. it is defined in a weighted projective $6$ space by $5$ Pfaffians of a $5 \times 5$ skew symmetric matrix $M$, and they are distinguished by their degrees $A^3 = (-K_X)^3$.
Below we give descriptions of $M$ and its Pfaffians $F_1,\dots,F_5$ for the $3$ families.

Codimension $3$ Fano $3$-folds $X \subset \mbP (1,5,6,7,8,9,10)$ of degree $1/42$:
\[
M =
\begin{pmatrix}
0 & a_6 & a_7 & a_8 & a_9 \\
& 0 & b_8 & b_9 & b_{10} \\
& & 0 & c_{10} & c_{11} \\
& & & 0 & d_{12} \\
& & & & 0
\end{pmatrix}
\hspace{1.5cm}
\begin{aligned}
F_1 &= a_6 c_{10} - a_7 b_9 + a_8 b_8  \\
F_2 & = a_6 c_{11} - a_7 b_{10} + a_9 b_8 \\ 
F_3 &= a_6 d_{12} - a_8 b_{10} + a_9 b_9 \\ 
F_4 &= a_7 d_{12} - a_8 c_{11} + a_9 c_{10} \\
F_5 &= b_8 d_{12} - b_9 c_{11} + b_{10} c_{10}
\end{aligned}
\]
Basket of singularities of $X$ are
\[
\frac{1}{2} (1,1,1), \ 
\frac{1}{3} (1,1,2), \ 
\frac{1}{5} (1,1,4), \ 
\frac{1}{5} (1,2,3), \ 
\frac{1}{7} (1,1,6). \ 
\]

Codimension $3$ Fano $3$-folds $X \subset \mbP (1,5,5,6,7,8,9)$ of degree $1/30$:
\[
M =
\begin{pmatrix}
0 & a_5 & a_6 & a_7 & a_8 \\
& 0 & b_7 & b_8 & b_9 \\
& & 0 & c_9 & c_{10} \\
& & & 0 & d_{11} \\
& & & & 0
\end{pmatrix}
\hspace{1.5cm}
\begin{aligned}
F_1 &= a_5 c_9 - a_6 b_8 + a_7 b_7  \\
F_2 & = a_5 c_{10} - a_6 b_9 + a_8 b_7 \\ 
F_3 &= a_5 d_{11} - a_7 b_9 + a_8 b_8 \\ 
F_4 &= a_6 d_{11} - a_7 c_{10} + a_8 c_9 \\
F_5 &= b_7 d_{11} - b_8 c_{10} + b_9 c_9
\end{aligned}
\]
Baskets of singularities of $X$ are
\[
\frac{1}{5} (1,1,4), \ 
2 \times \frac{1}{5} (1,2,3), \ 
\frac{1}{6} (1,1,5).
\]

Codimension $3$ Fano $3$-folds $X \subset \mbP (1,4,5,5,6,7,8)$ of degree $1/20$:
\[
M =
\begin{pmatrix}
0 & a_4 & a_5 & a_6 & a_7 \\
& 0 & b_6 & b_7 & b_8 \\
& & 0 & c_8 & c_9 \\
& & & 0 & d_{10} \\
& & & & 0
\end{pmatrix}
\hspace{1.5cm}
\begin{aligned}
F_1 &= a_4 c_8 - a_5 b_7 + a_6 b_6  \\
F_2 & = a_4 c_9 - a_5 b_8 + a_7 b_6 \\ 
F_3 &= a_4 d_{10} - a_6 b_8 + a_7 b_7 \\ 
F_4 &= a_5 d_{10} - a_6 c_9 + a_7 c_8 \\
F_5 &= b_6 d_{10} - b_7 c_9 + b_8 c_8
\end{aligned}
\]
Baskets of singularities of $X$ are
\[
\frac{1}{2} (1,1,1), \ 
\frac{1}{4} (1,1,3), \ 
2 \times \frac{1}{5} (1,1,4), \ 
\frac{1}{5} (1,2,3)_+.
\]
Here $a_i,b_i,c_i,d_i \in \mbC [x,y,\dots,v]$ are homogeneous polynomials of degree $i$.
The subscript $+$ of singularity means that $B^3 > 0$ for the singular point.
Thus the singular point $\frac{1}{5} (1,2,3)$ on a codimension $3$ Fano $3$-fold of degree $1/20$ is the unique singular point for which $B^3 > 0$.

The aim of this section is to prove Theorem \ref{mainthm} for these $3$ families which will follow from Propositions \ref{prop:lctPfaffnspt} and \ref{prop:lctPfaffsing} below.

Let $X$ be a codimension $3$ Fano $3$-folds of degree $A^3 \in \{1/42, 1/30, 1/20\}$.
We set $d = 1/A^3$.
As a generality condition for a member $X$ of these three families, we assume the following.

\begin{Cond}
\begin{enumerate}
\item $X$ is quasi-smooth.
\item $H_x$ is quasi-smooth.
\item The conditions given in \cite{AO} are satisfied.
\end{enumerate}
\end{Cond} 

\begin{Lem} \label{lem:Pfisol}
Let $\msp$ be a nonsingular point of $X$.
\begin{enumerate}
\item If $\msp \in H_x$, then $d A$ isolates $\msp$.
\item If $\msp \notin H_x$, then $a_3 A$ isolates $\msp$.
\end{enumerate}
\end{Lem}

\begin{proof}
Let $\mbP (1,a_1,\dots,a_6)$ be the ambient space of $X$ with $a_1 \le a_2 \le \cdots \le a_6$ and let $\pi \colon X \to \mbP (1,a_1,a_2,a_3)$ be the projection with coordinates $x,y,z,t$.
It is everywhere defined since $\Sigma := X \cap \Pi_{x,y,z,t} = \emptyset$ (see Example \ref{ex:Pfaffproj} below).
More precisely $\pi$ is a finite surjective morphism of degree $5$ although we do not need to know the degree of $\pi$.
Let $\msp = (\alpha_0\!:\!\alpha_1\!:\!\cdots\!:\!\alpha_6)$ be the coordinate.

Suppose that $\msp \in H_x$, that is, $\alpha_0 = 0$.
If $\alpha_1 \ne 0$, then the common zero locus of the sections in the set
\[
\{x, \alpha_1 z^{a_1} - \alpha_2 y^{a_2}, \alpha_1 t^{a_1} - \alpha_3 y^{a_3}\}
\]
coincides with $\pi^{-1} (\pi (\msp))$.
If $\alpha_1 = 0$ and $\alpha_2 \ne 0$, then the common zero locus of the sections in the set
\[
\{x,y, \alpha_2 t^{a_3} - \alpha_2 z^{a_3}\}
\]
coincides with $\pi^{-1} (\pi (\msp))$.
If $\alpha_1 = \alpha_2 = 0$, then $\alpha_3 \ne 0$ since $\Sigma = \emptyset$ and the common zero loci of the sections in the set $\{x,y,z\}$ coincides with $\pi^{-1} (\pi (\msp))$.
Thus $a_2 a_3 A$ isolates $\msp$ since $\pi^{-1} (\pi (\msp))$ is a finite set containing $\msp$.
The assertion (1) follows since $a_2 a_3 = d$.

Suppose that $\msp \notin H_x$, that is, $\alpha_0 \ne 0$.
The the common zero of the sections in the set
\[
\{ \alpha_0 y - \alpha_1 x^{a_1}, \alpha_0 z - \alpha_2 x^{a_2}, \alpha_0 t - \alpha_3 x^{a_3}\}
\]
coincides with $\pi^{-1} (\pi (\msp))$.
Thus $a_3 A$ isolates $\msp$ and the proof is completed.
\end{proof}

\begin{Ex} \label{ex:Pfaffproj}
We explain the assertion $X \cap \Pi_{x,y,z,s} = \emptyset$ for $X$ of degree $1/42$.
We set $\Pi = \Pi_{x,y,z,s} \cong \mbP (8_t,9_u,10_v)$.
Since $\msp_t, \msp_u, \msp_v \notin X$, we have $t^2 \in F_1$, $u^2, v^2 \in F_2$.
This implies that the coefficients of $t$ in $a_8$ and $b_8$, $u$ in $a_9, b_9$ and $v$ in $b_{10}, c_{10}$ are non-zero.
Moreover $a_6,a_7,c_{11}$ and $d_{12}$ vanish along $\Pi$.
It follows that $F_2$ and $F_4$ vanish along $\Pi$ and 
\[
F_1|_{\Pi} = \alpha t^2, \ 
F_3|_{\Pi} = \beta u^2, \ 
F_5|_{\Pi} = \gamma v^2,
\]
for some non-zero $\alpha,\beta,\gamma \in \mbC$.
This shows $X \cap \Pi = \emptyset$.
Proofs for the other families are similar.
\end{Ex}

\begin{Prop} \label{prop:lctPfaffnspt}
Let $X$ be a birationally rigid codimension $3$ Fano $3$-fold and $\msp$ a nonsingular point of $X$.
Then $\mult_{\msp} (D) \le 1$ for any effective $\mbQ$-divisor $D \sim_{\mbQ} A$.
\end{Prop}

\begin{proof}
Suppose that $\msp \in H_x$.
We have $\mult_{\msp} (H_x) = 1$, and by Lemma \ref{lem:Pfisol}, $dA$ isolates $\msp$, where $d = 1/A^3$.
Thus the assertion follows from Lemma \ref{lem:exclG} by setting $S = H_x$,

Suppose $\msp \notin H_x$.
We may assume $\msp = \msp_x$.
By Lemma \ref{lem:Pfisol}, $a_3 A$ isolates $\msp$.
If $X$ is of degree $1/42$, then we set $S_1 = H_y$ and $S_2 = H_z$.
We have $a_3 = 7$ and $6 \cdot 7 \cdot A^3 = 1$.
If $X$ is of degree $1/30$, then the linear system $|\mcI_{\msp} (5A)|$ is a pencil and let $S_1, S_2$ be distinct members of the pencil.
We have $a_3 = 6$ and $5 \cdot 6 \cdot A^3 = 1$.
The assertion follows from Lemma \ref{lem:exclG}.
Suppose that $X$ is of degree $1/20$.
We first show that $\mult_{\msp} (H_y) \le 4$.
Assume to the contrary that $\mult_{\msp} (H_y) > 5$. 
Then, by Lemma \ref{lem:highmultdiv}, there is an effective divisor $R \ne H_y$ such that $R \in |\mcI_{\msp}^2 (m A)|$ for some $m \le 8$.
Since $a_3 A = 5 A$ isolates $\msp$, we can take $T \in |\mcI_{\msp}^k (5 k A)|$ which does not contain any component of the effective $1$-cycle $H_y \cdot R$, hence we have
\[
8 k \ge 20k m A^3 = H_y \cdot R \cdot T \ge 4 \cdot 2 \cdot k = 8 k.
\]
This is a contradiction and we have $\mult_{\msp} (H_y) \le 4$.
Now we can apply Lemma \ref{lem:exclG} for $S = H_y$ and the assertion immediately follows since $4 \cdot 5 \cdot A^3 = 1$.
\end{proof}

In the following we consider codimension $3$ Fano $3$-folds $X$ of degree $1/20$. 
We define $L_{xy} = H_x \cap H_y$.
Let $\msp \in X$ be the singular point of type $\frac{1}{5} (1,2,3)$.
Let $\rho = \rho_{\msp} \colon \breve{X}_{\msp} \to X$ be the index cover of $\msp \in X$ and $\breve{\msp}$ the preimage of $\msp$.
We set $\breve{L}_{xy} = \rho^*H_x \cap \rho^*H_y$.

\begin{Lem} \label{lem:Pf20L}
The scheme $L_{xy}$ is an irreducible and reduced curve, and $\mult_{\breve{\msp}} (\breve{L}_{xy}) = 2$, where $\msp \in X$ is the singular point of type $\frac{1}{5} (1,2,3)$.
\end{Lem}

\begin{proof}
Set $\Pi = \Pi_{x,y} \cong \mbP (5_z,5_s,6_t,7_u,8_v)$.
By a suitable choice of coordinates, we can write $M|_{\Pi}$ and $F_i|_{\Pi}$ as follows:
\[
M|_{\Pi} =
\begin{pmatrix}
0 & 0 & z & t & u \\
& 0 & \alpha t & \beta u & v \\
& & 0 & \gamma v & 0 \\
& & & 0 & q \\
& & & & 0
\end{pmatrix}
\hspace{1.5cm}
\begin{aligned}
F_1|_{\Pi} &= - \beta u z + \alpha t^2  \\
F_2|_{\Pi} & = - v z + \alpha u t \\ 
F_3|_{\Pi} &= - v t + \beta u^2 \\ 
F_4|_{\Pi} &= z q + \gamma v u \\
F_5|_{\Pi} &= \alpha t q + \gamma v^2
\end{aligned}
\]
Here $\alpha,\beta,\gamma \in \mbC$ and $q = q (z,s)$ is a quadratic form in $z,s$.
We see that $s^2 \in q$ and $\msp = \msp_s$.
By quasi-smoothness, $X$ does not pass through $\msp_t, \msp_u$ and $\msp_v$, which implies that none of $\alpha,\beta,\gamma$ is non-zero.

To prove that $L = L_{xy} \subset \Pi$ is irreducible and reduced, we work on the open subset $U_z \subset X \cap \Pi$ on which $z \ne 0$.
By setting $z = 1$, $L \cap U_z$ is defined in $\mbA^4_{s,t,u,v}/\mbZ_5$ by the equations
\[
-\beta u + \alpha t^2 = - v + \alpha u t = - vt + \beta u^2 = q (1,s) + \gamma v u = \alpha t q (1,s) + \gamma v^2 = 0.
\]
Eliminating $u, v$ and redundant defining equations, we see that $L \cap U_z$ is isomorphic to the scheme defined in $\mbA^2_{s,t}/\mbZ_5$ by the equation
\[
q (1,s) - \frac{\alpha^3 \gamma}{\beta^2} t^5 = 0.
\]
Since $s^2 \in q (1,s)$ and $\alpha \beta \gamma \ne 0$, $L \cap U_{z_0}$ is irreducible and reduced.
Thus $L$ is irreducible and reduced since $L \cap H_z = \{\msp\}$.

To prove that $\mult_{\breve{\msp}} (\breve{L}_{\msp}) = 2$, we work on the open subset $U := U_s \subset X \cap \Pi$ on which $s \ne 0$.
The index $1$ cover $\breve{U}$ of $\msp \in U$ is defined in $\mbA^4_{z,t,u,v}$ by the equations
\[
- \beta u z + \alpha t^2 = - v z + \alpha u t = - v t + \beta u^2 = z q_1 + \gamma v u = \alpha t q_1 + \gamma v^2 = 0,
\]
where $q_1 = q (z,1)$.
Re-scaling $s$, we may assume that the constant term of $q_1$ is $1$.
From now on we think of $\breve{U}$ as a germ $\breve{\msp} \in \breve{U}$ so that $q_1$ is invertible on $\breve{U}$.
Then we have $z = - \gamma v u + f$ and $t = - \gamma v^2 + g$, where $f, g \in \mcI_{\breve{\msp}}^3$.
By eliminating $z$ and $t$, the defining equation of $\breve{L}$ in $\mbA^2_{u,v}$ modulo $\mcI_{\breve{\msp}}$ is
\[
\gamma v^3 + \beta u^2 = 0.
\]
Thus $\mult_{\breve{\msp}} (\breve{L}) = 2$ since $\beta \ne 0$, and the proof is completed.
\end{proof}

\begin{Prop} \label{prop:lctPfaffsing}
Let $X$ be a codimension $3$ Fano $3$-fold and $\msp \in X$ a singular point.
Then $\lct_{\msp} (X) \ge 1$.
\end{Prop}

\begin{proof}
Let $\msp$ be a singular point other than the singular point of type $\frac{1}{5} (1,2,3)$ on $X$ of degree $1/20$.
Then $B^3 \le 0$ and it is proved in \cite{AO} that $(-K_X)^2 \notin \bNE (Y)$.
Moreover, we have $\tilde{H}_x \sim B$.
Thus the assertion follows from Lemma \ref{lem:singptNE}.
It remains to consider the singular point $\msp$ of type $\frac{1}{5} (1,2,3)$ on $X$ of degree $1/20$.
We will apply Lemma \ref{lem:singptNE} for $S_1 = H_x \sim_{\mbQ} A$ and $S_2 = H_y \sim_{\mbQ} 4 A$.
By \cite[Lemma 6.8]{AO}, $B = -K_Y$ is nef and big, and we have $\ord_E (H_x) = 1/5$ (see \cite[Section 6.4]{AO}).
By Lemma \ref{lem:Pf20L}, we have $\mult_{\breve{\msp}} (\breve{L}_{xy}) = 2$.
Finally, we have $5 \cdot 1 \cdot 4 \cdot  A^3 = 1$.
Therefore the assumptions in Lemma \ref{lem:singptNE} are satisfied and the assertion follows.
\end{proof}


\end{document}